\documentclass[11pt]{article}

\usepackage{authblk}
\usepackage{amssymb}
\usepackage{amsmath}
\usepackage{bm}
\usepackage{amsthm, color, fullpage}
\allowdisplaybreaks
\def\og{\leavevmode\raise.3ex\hbox{$\scriptscriptstyle\langle\!\langle$~}}
\def\fg{\leavevmode\raise.3ex\hbox{~$\!\scriptscriptstyle\,\rangle\!\rangle$}}

\newcommand{\beqa}{\begin{eqnarray}}
\newcommand{\eeqa}[1]{\label{#1}\end{eqnarray}}
\newcommand{\beq}{\begin{equation}}
\newcommand{\eeq}[1]{\label{#1}\end{equation}}

\newcommand\ep{\varepsilon}

\def\demifleche{\rightharpoonup}

\def\*fleche{\buildrel *\over\demifleche}

\def\tol2{\buildrel\hbox{$L^2$}\over\longrightarrow}

\def\toto{\leaders\hbox to 5mm{\hfil.\hfil}\hfill}

\newtheorem{theorem}{\bf Theorem}[section]

\newtheorem{corollary}[theorem]{\bf Corollary}
\newtheorem{lemma}[theorem]{\bf Lemma}

\newtheorem{proposition}[theorem]{\bf Proposition}
\newtheorem{definition}{\bf Definition}
\newtheorem{assumption}[theorem]{Assumption}

\begin{document}

\title{\sc Operator-norm convergence estimates for elliptic homogenisation problems on periodic singular structures}

\author[1]{Kirill Cherednichenko}
\author[1]{Serena D'Onofrio}
\affil[1]{Department of Mathematical Sciences, University of Bath, Claverton Down, Bath, BA2 7AY, United Kingdom}

\maketitle
\vspace{-6mm}
\begin{center}
  \textsl{To the fond memory of Vasily Vasil'evich Zhikov}
 \vspace{4mm} 
\end{center}

\begin{abstract}
 
For a an arbitrary periodic Borel measure $\mu$, we prove order $O(\varepsilon)$ operator-norm resolvent estimates for the solutions to scalar elliptic problems in $L^2({\mathbb R}^d, d\mu^\varepsilon)$ with $\varepsilon$-periodic coefficients, $\varepsilon>0.$  Here $\mu^\varepsilon$ is 
the measure obtained by $\varepsilon$-scaling of $\mu.$ Our analysis includes both the case of a measure
%periodic measure  
%In addition to the classical case of the Lebesgue measure, 
%$\mu$ 
%that is 
absolutely continuous with respect to the standard Lebesgue measure and the case of ``singular" periodic structures (or ``multisctructures''), when $\mu$ is supported by lower-dimensional manifolds.

 \vskip 0.4cm


\end{abstract}

\section{Introduction}
\label{intro}

The goal of the present work is to prove order-sharp norm-resolvent convergence estimates for partial differential operators with periodic rapidly oscillating coefficients for a  wide class of underlying periodic measures. The results on norm-resolvent convergence in homogenisation for the ``classical" problem concerning the case of Lebesgue measure go back to the works \cite{Sevost'yanova}, \cite{Zhikov1989}, where the asymptotic analysis of the Green functions of the corresponding problems is carried out, which were followed by the operator-theoretic approach of \cite{BS}. An alternative approach, based on the uniform power-series asymptotic analysis of the fibre operators in the associated direct integral, was recently developed in  \cite{ChCoARMA}. In the present work we adopt the overall strategy of the latter work, in the setting of an arbitrary periodic Borel measure.  As was pointed out in \cite{Zhikov2000}, and subsequently discussed in more detail in \cite{Zhikov_Note}, \cite{Zhikov_Pastukhova_Bloch}, the analysis of related elliptic problems requires a careful description of the property of (weak) differentiability 
of functions square integrable with respect to a general Borel measure. In what follows we briefly introduce 
the tools we employ, namely the Sobolev spaces of quasiperiodic functions with respect to an arbitrary Borel measure (Section \ref{quasiperiodic}) and the Floquet transform (Section \ref{Floquet}). In Section \ref{main_result} we formulate and prove our main result (Theorem \ref{main_theorem}). 
%%All functions spaces that we use are defined over the field ${\mathbb C}$ of complex numbers. 
Throughout the paper, for vectors $a, b\in{\mathbb C}^3$ we denote by $a\cdot b$ their standard (sesquilinear) Euclidean inner product, and define all function spaces over the field ${\mathbb C}.$

Consider a $Q$-periodic, $Q:=[0,1)^d,$ Borel measure $\mu,$ in ${\mathbb R}^d$ such that $\mu(Q)=1,$ and for each $\varepsilon>0$ define an $\varepsilon$-periodic measure $\mu^\varepsilon$ by the formula $\mu^\varepsilon(B)=\varepsilon^d\mu(\varepsilon^{-1}B)$ for all Borel sets $B\subset{\mathbb R}^d,$ $d\in{\mathbb N}.$ 
%%%%{\color{red} Discuss the density of smooth functions in $L^2$ and the compactness of embedding of $H^1$ into $L^2.$}
In the present work we study the asymptotic behaviour, as $\varepsilon\to0,$ of the solutions $u=u^\varepsilon$ to the problems 
\begin{equation}
-\nabla\cdot A(\cdot/\varepsilon)\nabla u+u=f,\qquad f\in L^2({\mathbb R}^d,d\mu^\varepsilon),\qquad \varepsilon>0,
\label{whole_space_eq}
\end{equation}
where $A$ is a positive bounded $Q$-periodic $\mu$-measurable real-valued matrix function. We aim to prove operators-norm estimates between $u^\varepsilon$ and the solution to the homogenised equation 
\begin{equation}
-\nabla\cdot A^{\rm hom}\nabla u^0+u^0=f,\qquad f\in L^2({\mathbb R}^d,d\mu^\varepsilon),
\label{u0_eq}
\end{equation}
with a constant matrix $A^{\rm hom},$ {\it i.e.} uniform estimates of the form
\[
\Vert u-u^0\Vert_{L^2({\mathbb R}^d, d\mu^\varepsilon)}\le C\varepsilon\Vert f\Vert_{L^2({\mathbb R}^d,d\mu^\varepsilon)},
\] 
where $C>0$ is independent of $f,$ $\varepsilon.$ 

Solutions to (\ref{whole_space_eq}) are understood as a pair $(u, \nabla u)$ in the space $H^1({\mathbb R}^d, d\mu^\varepsilon),$ defined ({\it cf.} \cite{Zhikov_Pastukhova_Bloch}) as the closure of the set $\{(\psi, \nabla\psi), \psi\in C_0^\infty({\mathbb R}^d)\}$ in the norm of 
$L^2({\mathbb R}^d, d\mu^\varepsilon)\oplus\bigl[L^2({\mathbb R}^d, d\mu^\varepsilon)\bigr]^d.$ For $f\in L^2({\mathbb R}^d, d\mu^\varepsilon),$ we  say that $(u, \nabla u)\in H^1({\mathbb R}^d, d\mu^\varepsilon)$ is a solution to (\ref{whole_space_eq}) if  
\begin{equation}
\int_{{\mathbb R}^d}A(\cdot/\varepsilon)\nabla u\cdot{\nabla\psi}\,d\mu^\varepsilon+\int_{{\mathbb R}^d}u\overline{\psi}\,d\mu^\varepsilon=\int_{{\mathbb R}^d}f\overline{\psi}\,d\mu^\varepsilon\qquad \forall 
(\psi,\nabla \psi)\in H^1({\mathbb R}^d, d\mu^\varepsilon).
% \psi\in C^\infty_0({\mathbb R}^d).
\label{whole_space_eq_weak}
\end{equation}
Note that for each $\varepsilon>0$ the left-hand side of (\ref{whole_space_eq_weak}) is an equivalent inner product on $H^1({\mathbb R}^d, d\mu^\varepsilon),$ and its right-hand side is a linear bounded functional on $H^1({\mathbb R}^d, d\mu^\varepsilon).$
%with respect to the norm associated with this inner product. 
Invoking the Riesz representation theorem (see {\it e.g.} \cite[p.\,32]{Birman_Solomjak}) yields the existence and uniqueness of solution to (\ref{whole_space_eq}).

In what follows we study the resolvent of the operator ${\mathcal A}^\varepsilon$ with domain
\begin{align}
&{\rm dom}({\mathcal A}^\varepsilon)=\biggl\{u\in L^2({\mathbb R}^d, d\mu^\varepsilon):\ \exists\, \nabla u\in\bigl[L^2({\mathbb R}^d, d\mu^\varepsilon)\bigr]^d\ {\rm such\ that}\nonumber\\[0.5em] 
%: \bigl(e_\varkappa u, \nabla (e_\varkappa u)\bigr)\in H_\varkappa^1,\\[0.5em] 
 &\int_{{\mathbb R}^d}A(\cdot/\varepsilon)\nabla u\cdot{\nabla\psi}\,d\mu^\varepsilon+\int_{{\mathbb R}^d}u\overline{\psi}\,d\mu^\varepsilon=\int_{{\mathbb R}^d}f\overline{\psi}\,d\mu^\varepsilon\qquad\forall(\psi,\nabla \psi)\in H^1({\mathbb R}^d, d\mu^\varepsilon)\nonumber\\[0.5em]
 &
 %\psi\in C_0^\infty({\mathbb R}^d)
\qquad\qquad\qquad\qquad\qquad\qquad\qquad\qquad\qquad{\rm for\ some}\ f\in L^2({\mathbb R}^d, d\mu^\varepsilon)\biggr\}.\label{Ae_domain}
\end{align}
defined by the formula ${\mathcal A}^\varepsilon u=f-u$ whenever $f\in L^2({\mathbb R}^d, d\mu^\varepsilon)$ and $u\in{\rm dom}({\mathcal A}^\varepsilon)$ are related as in (\ref{Ae_domain}). Note that while in general for a given $u\in L^2({\mathbb R}^d, d\mu^\varepsilon)$ there may be more than one element  $(u,\nabla u)\in H^1({\mathbb R}^d, d\mu^\varepsilon),$ 
%since for $f=0$ one has $u=0,$ $\nabla u=0$ (which is equivalent to 
the uniqueness of solution to (\ref{whole_space_eq}) 
%obtain above), 
implies that for each function $u\in{\rm dom}({\mathcal A}^\varepsilon)$ there is exactly one gradient $\nabla u$ such that the identity in (\ref{Ae_domain}) holds. 

Clearly, the operator ${\mathcal A}^\varepsilon$ is symmetric. By an argument similar to \cite[Section 7.1]{Zhikov2000}, we infer that ${\rm dom}({\mathcal A}^\varepsilon)$ is dense in $L^2({\mathbb R}^d, d\mu^\varepsilon).$ Indeed, it follows from (\ref{Ae_domain}) that if $f\in L^2({\mathbb R}^d, d\mu^\varepsilon),$ and $u, v\in{\rm dom} ({\mathcal A}^\varepsilon)$ are such that ${\mathcal A}^\varepsilon u+u=f,$ ${\mathcal A}^\varepsilon v+v=u,$ then
\begin{equation}
\int_{{\mathbb R}^d}f\overline{v}=\int_{{\mathbb R}^d}\vert u\vert^2.
\label{RH_id}
\end{equation}
The identity (\ref{RH_id}) implies that if $f$ is orthogonal to ${\rm dom}({\mathcal A}^\varepsilon)$ then $u=0,$ and hence $f=0.$ Furthermore, 
%the density of ${\rm dom}({\mathcal A}^\varepsilon)$ in $L^2({\mathbb R}^d, d\mu^\varepsilon)$ implies that 
${\mathcal A}^\varepsilon$ is self-adjoint. Indeed, suppose that  $w\in {\rm dom}\bigl(({\mathcal A}^\varepsilon)^*\bigr)\subset L^2({\mathbb R}^d, d\mu^\varepsilon),$ so for some  $g\in L^2({\mathbb R}^d, d\mu^\varepsilon)$ one has 
\[
\int_{{\mathbb R}^d}({\mathcal A}^\varepsilon u)\overline{w}\,d\mu^\varepsilon=\int_{{\mathbb R}^d}u\overline{g}\,d\mu^\varepsilon\qquad \forall u\in {\rm dom}({\mathcal A}^\varepsilon).
%${\rm dom}({\mathcal A}^\varepsilon)$ is dense in $$
\]
Consider the solution $v$ to the problem 
\[
{\mathcal A}^\varepsilon v+v=g+w.
\]
Then for all $u\in {\rm dom}({\mathcal A}^\varepsilon)$ one has 
\[
\int_{{\mathbb R}^d}({\mathcal A}^\varepsilon u+u)\overline{w}=\int_{{\mathbb R}^d}u\overline{(g+w)}=\int_{{\mathbb R}^d}u\overline{({\mathcal A}^\varepsilon v+v)}=\int_{{\mathbb R}^d}({\mathcal A}^\varepsilon u+u)\overline{v},
\]
where we use the fact that ${\mathcal A}^\varepsilon$ is symmetric and $u, v\in {\rm dom}({\mathcal A}^\varepsilon).$ Since ${\mathcal A}^\varepsilon u+u$ is an arbitrary element of $L^2({\mathbb R}^d, d\mu^\varepsilon),$ it follows that $w=v,$ and in particular 
$w\in {\rm dom}({\mathcal A}^\varepsilon).$

%${\rm dom}\bigl((A^\varepsilon)^*\bigr)\subset{\rm dom}({\mathcal A}^\varepsilon),$ from which the claim follows.

Similarly, we define the operator ${\mathcal A}^{\rm hom}$ associated with the problem (\ref{u0_eq}), so that (\ref{u0_eq}) holds if and only if $u^0=({\mathcal A}^{\rm hom}+I)^{-1}f.$
% the definition of ${\rm dom}({\mathcal A}^\varepsilon).$

All gradients, integrals and differential operators below, unless indicated explicitly otherwise, are understood appropriately with respect to the measure $\mu.$ Whenever we write $\int_Q,$ we imply integration with respect to the measure $\mu$ and interchangeably use the notation and $L^2(Q, d\mu)$ and $L^2(Q)$ for the Lebesgue space of functions that are square integrable on $Q$ with respect to $\mu.$ 
%Similarly, $H_\#^1$ denotes the Sobolev space of $Q$-periodic functions, with respect to the measure $\mu,$ see the definition below. 
Throughout the paper we use the notation $e_\varkappa$ for the exponent $\exp({\rm i}\varkappa\cdot y),$ $y\in Q,$ $\varkappa\in[-\pi, \pi)^d,$ and a simlilar notation  $e_\theta$ for the exponent $\exp({\rm i}\theta\cdot x),$ $x\in {\mathbb R}^d,$ $\theta\in\varepsilon^{-1}[-\pi, \pi)^d.$ 
We denote by $C_\#^\infty$ the set of $Q$-periodic functions in $C^\infty({\mathbb R}^d),$ and $\partial_j\phi,$ $\nabla \phi,$ $\nabla(e_\varkappa\phi)$ $\nabla (e_{\varepsilon\theta}\phi)$ stand for the classical derivatives and gradients of smooth functions $\phi,$ $e_\varkappa\phi,$ $e_{\varepsilon\theta}\phi.$ The symbol ``$:=$" stands is used to denote the expression on the right-hand side of the symbol by its left-hand side.

\section{Sobolev spaces of quasiperiodic functions}
\label{quasiperiodic}

The material of this section applies to an arbitrary Borel measure $\mu$ on $Q.$ The following definition is motivated by \cite[Section 3.1]{Zhikov2000}, \cite{Zhikov_Note}. 
\begin{definition}
\label{Sobolev_definition}
For each $\varkappa\in[-\pi, \pi)^d:=Q'$ we define the space $H_\varkappa^1$ as the closure, with respect to the natural norm of the direct sum
$L^2(Q)\oplus\bigl[L^2(Q)\bigr]^d,$
%\[
%\Vert\cdot\Vert_\varkappa:=\biggl(\int_Q\vert\cdot\vert^2d\mu+\int_Q\vert\nabla\cdot\vert^2d\mu\biggr)^{1/2}
%\]
of the set 
$
\bigl\{\bigl(e_\varkappa\phi, \nabla (e_\varkappa\phi)\bigr): \phi\in C_\#^\infty\bigr\}.$ 
%where the gradients $\nabla (e_\varkappa\phi)$ are understood in the classical sense. 
We use the notation $H_\#^1$ for the space $H_\varkappa^1,$ $\varkappa=0.$ For $(u,v)\in H_\varkappa^1$ we keep the usual notation $\nabla u$ for the second element $v$ in the pair.
\end{definition}

As discussed in \cite{Zhikov2000}, \cite{Zhikov_Note}, \cite{Zhikov_Pastukhova_Bloch},
%the above works by V. V. Zhikov, 
there may be different elements in $H_\varkappa^1$ whose first components coincide. Indeed, for any $(u, v)\in H_\varkappa^1$  and a vector function $w$ obtained as the limit in 
$\bigl[L^2(Q)\bigr]^d$ of the classical gradients $\nabla(e_\varkappa\phi_n)$ for a sequence $\phi_n\in C_\#^\infty$ converging to zero in $L^2(Q)$ (``gradient of zero''), the pair 
$(u, v+w)$ is also an element of $H_\varkappa^1.$ Furthermore, there is a natural one-to-one mapping between $H_\varkappa^1$ and $H_\#^1$: for any element $(u, v)\in H_\varkappa^1$ the pair $\bigl(\overline{e_\varkappa}u,\,\overline{e_\varkappa}(v-{\rm i}u\varkappa)\bigr)$ is an element of $H_\#^1$ and for all $(\widetilde{u}, \widetilde{v})\in H_\#^1$ one has $\widetilde{v}=\overline{e_\varkappa}(v-{\rm i}u\varkappa)$ for some $(u, v)\in H_\varkappa^1.$ In view of this, for $(\widetilde{u}, \widetilde{v})\in H_\#^1$ we often write $\widetilde{v}=\nabla\widetilde{u}=\overline{e_\varkappa}\nabla(e_\varkappa\widetilde{u})-{\rm i}\widetilde{u}\varkappa,$ where either $\nabla\widetilde{u}$ or $\nabla(e_\varkappa\widetilde{u})$ is defined up to a gradient of zero.
%, {\it i.e.} a limit in $\bigl[L^2(Q)\bigr]^d$ of the gradients $\nabla \phi_n$ for some sequence $\phi_n\in C^\infty_\#$ converging to zero in $L^2(Q).$

%pair $(e_\varkappa u, v+{\rm u}\varkappa)$

%and we also write $(u, v)=(e_\varkappa\widetilde{u}, e_\varkappa\nabla )$

Suppose that $A\in\bigl[L^\infty(Q, d\mu)\bigr]^{d\times d}$ is a pointwise positive and symmetric real-valued matrix function such that 
$A^{-1}\in \bigl[L^\infty(Q, d\mu)\bigr]^{d\times d},$ and for each $\varkappa\in Q'$ consider the operator ${\mathcal A}_\varkappa$ with domain ({\it cf.} (\ref{Ae_domain}))
\begin{align}
{\rm dom}({\mathcal A}_\varkappa)&=\biggl\{u\in L^2(Q):\ \exists\, \nabla(e_\varkappa u)\in\bigl[L^2(Q)\bigr]^d\ {\rm such\ that}\nonumber\\[0.5em] 
%: \bigl(e_\varkappa u, \nabla (e_\varkappa u)\bigr)\in H_\varkappa^1,\\[0.5em] 
 &\int_QA\nabla(e_\varkappa u)\cdot{\nabla(e_\varkappa\varphi)}+\int_Qu\overline{\varphi}=\int_QF\overline{\varphi}\ \ \ \forall\varphi\in C_\#^\infty\ \ \ {\rm for\ some}\ F\in L^2(Q)\biggr\},\label{A_kappa}
\end{align}
defined by the formula ${\mathcal A}_\varkappa u=F-u$ whenever $F\in L^2(Q)$ and $u\in{\rm dom}({\mathcal A}_\varkappa)$ are related as described in the definition of ${\rm dom}({\mathcal A}_\varkappa).$ Notice that by the definition of $H^1_\varkappa,$ the set $C_\#^\infty$ of test functions in the identity in (\ref{A_kappa}) can be equivalently replaced by $H^1_\varkappa.$ As discussed in the previous section for the case of operator ${\mathcal A}^\varepsilon,$ since for $F=0$ one has $u=0,$ $\nabla (e_\varkappa u)=0,$ there is exactly one gradient $\nabla (e_\varkappa u)$ for which (\ref{A_kappa}) holds. Also, by an argument similar to the case of ${\mathcal A}^\varepsilon,$
%\cite[Section 7.1]{Zhikov2000}, 
the domain  ${\rm dom}({\mathcal A}_\varkappa)$ is dense in $L^2(Q)$ and ${\mathcal A}_\varkappa$ is self-adjoint.
% the set of the first components of elements of $H_\#^1,$ which we will henceforth identify with $H_\#^1:$ 

%if $w\in({\rm dom}({\mathcal A}_\varkappa))^\perp$ then 

% Clearly, in the sense of the 

In what follows,  we identify $H_\#^1$ and the the set of the first components of its elements, bearing in mind that the gradient of a function in $H_\#^1$  may not be unique. We also denote by $H_{\#,0}^1$ the (closed) subspace of $H_\#^1$ consisting of functions with zero $\mu$-mean over $Q.$ 
%of $H_\#^1,$ which we will henceforth identify with $H_\#^1:

\section{Floquet transform}

\label{Floquet}

In this section we define a representation for functions in $L^2({\mathbb R}^d, d\mu)$ unitarily equivalent to ``Gelfand transform'', introduced in \cite{Gelfand} for the case of the Lebesgue measure. The properties of the Gelfand transform with respect to the measure $\mu$ are discussed in detail in \cite{Zhikov_Pastukhova_Bloch}, and here we give the definition of its Floquet version as well as the key property concerning the equation
(\ref{whole_space_eq}). We first define a ``scaled" version of the Floquet transform ({\it cf.} \cite{ChCoARMA}).

%the definition an the 
%with respect to the measure $\mu.$

\begin{definition}
For $\varepsilon>0$ and $u\in C^\infty_0({\mathbb R}^d)$, the $\varepsilon$-Floquet transform ${\mathcal F}_\varepsilon u$ is the function  
\[
%{\mathcal F}: u\mapsto 
({\mathcal F}_\varepsilon u)(z, \theta)=\biggl(\frac{\varepsilon}{2\pi}\biggr)^{d/2}\sum_{n\in Z^{d}}u(z+\varepsilon n)\exp(-{\rm i}\varepsilon n\cdot\theta),\quad z\in\varepsilon Q,\ \theta\in\varepsilon^{-1}Q'.
\]
\end{definition}
The mapping ${\mathcal F}_\varepsilon$ preserves the norm, in the sense that
\[
 \Vert {\mathcal F}_\varepsilon u\Vert_{L^2(\varepsilon^{-1}Q'\times\varepsilon Q, d\theta\times d\mu^\varepsilon)}=\Vert u\Vert_{L^2({\mathbb R}^d, d\mu^\varepsilon)},
\]
and
%in $L^2({\mathbb R}^d)$ the and 
it can therefore be extended to an isometry ${\mathcal F}_\varepsilon: L^2({\mathbb R}^d, d\mu^\varepsilon)\mapsto L^2(\varepsilon^{-1}Q'\times\varepsilon Q, d\theta\times d\mu^\varepsilon),$ for which we use the same term $\varepsilon$-Floquet transform. Note that the inverse of 
%${\mathcal T}_\varepsilon$ and 
${\mathcal F}_\varepsilon$ is given by 
\begin{equation}
({\mathcal U}_\varepsilon g)(z)=\biggl(\frac{\varepsilon}{2\pi}\biggr)^{d/2}\int_{\varepsilon^{-1}Q'}g(\theta,z)\,d\theta,\quad z\in {\mathbb R}^d,\qquad g\in L^2(\varepsilon^{-1}Q'\times\varepsilon Q, d\theta\times d\mu^\varepsilon),
\label{inversion}
\end{equation}
where for each $\theta\in\varepsilon^{-1}Q'$ the function $g\in L^2(\varepsilon^{-1}Q'\times \varepsilon Q, d\theta\times d\mu^\varepsilon)$ is extended as $\theta$-quasiperiodic function to the whole of ${\mathbb R}^d$ so that 
%{\it i.e.}
\[
g(\theta,z)=\widetilde{g}(\theta, z)\exp({\rm i}z\cdot\theta),\quad z\in{\mathbb R}^d,\qquad \widetilde{g}(\theta,\cdot)\ \ \ \varepsilon Q{\text{\rm -periodic.}}
\]
Indeed, for all such functions $g$ the right-hand side 
%of the second equality in 
(\ref{inversion}) is well defined and returns a function in $L^2({\mathbb R}^d),$ {\it cf.} \cite{Zhikov_Pastukhova_Bloch}:
%\begin{align*}
\begin{equation*}
\bigl\Vert{\mathcal U}_\varepsilon g\bigr\Vert_{L^2({\mathbb R}^d)}^2=\sum_{n\in{\mathbb Z}^d}\bigl\Vert({\mathcal U}_\varepsilon g)(\cdot+\varepsilon n)\bigr\Vert_{L^2(\varepsilon Q)}^2
%\\
=\sum_{n\in{\mathbb Z}^d}\bigl\Vert\widehat{g}_n\bigr\Vert_{L^2(\varepsilon Q)}^2=\int_{\varepsilon Q}\int_{\varepsilon^{-1}Q'}\bigl\vert({\mathcal U}_\varepsilon g)(\cdot, \theta)\bigr\vert^2d\theta\,d\mu^\varepsilon,
%=
%\bigr\Vert_{L^2(\varepsilon^{-1}Q',d\mu)}
%\biggl(\frac{\varepsilon}{2\pi}\biggr)^{d/2}
%$hL^2(Q'\times Q, dx\times d\mu),$
%\end{align*}
\end{equation*}
where, for each $z\in \varepsilon Q,$
\[
\widehat{g}_n(z):=\biggl(\frac{\varepsilon}{2\pi}\biggr)^{d/2}\int_{\varepsilon^{-1}Q'}g(\theta, z)\exp({\rm i}\varepsilon n\cdot\theta)d\theta,\qquad n\in{\mathbb Z}^d,
\]
are the Fourier coefficients of the $\varepsilon^{-1}Q'$-periodic function $g(\cdot, z).$ Since the image of ${\mathcal U}_\varepsilon$ contains $C_0^\infty({\mathbb R}^d)$ and for all
$u\in C_0^\infty({\mathbb R}^d)$ one has $u={\mathcal U}_\varepsilon{\mathcal F}_\varepsilon u,$ it follows that ${\mathcal F}_\varepsilon$ is one-to-one and thus, unitary.

Combining the $\varepsilon$-Floquet transform and the unitary scaling transform 
\begin{align*}
{\mathcal T}_\varepsilon h(\theta,y)&:=\varepsilon^{d/2}h(\theta, \varepsilon y),\qquad \theta\in\varepsilon^{-1}Q',\ y\in Q,\qquad \forall h\in L^2(\varepsilon^{-1}Q'\times\varepsilon Q, 
d\theta\times d\mu^\varepsilon),\\[0.5em]
({\mathcal T}_\varepsilon^{-1} h)(\theta, z)&=\varepsilon^{-d/2}h(\theta, z/\ep),\quad \theta\in\varepsilon^{-1}Q',\ z\in\varepsilon Q,\qquad \forall h\in L^2(\varepsilon^{-1}Q'\times Q, d\theta\times d\mu),
\end{align*}
%where in the second formula the function $h$ is extended by $Q$
we obtain a representation for the operator ${\mathcal A}^\varepsilon,$ as follows. 

\begin{proposition}
\label{Floquet_resolvent}
For each $\varepsilon>0$ the operator ${\mathcal A}^\varepsilon$ is unitarily equivalent to the direct integral of the family ${\mathcal A}_{\varepsilon\theta},$ $\theta\in\varepsilon^{-1}Q',$ namely 
\[
({\mathcal A}^\varepsilon+I)^{-1}={\mathcal F}_\varepsilon^{-1}{\mathcal T}_\varepsilon^{-1}\int_{\varepsilon^{-1}Q'}^\oplus e_{\varepsilon\theta}(\varepsilon^{-2}{\mathcal A}_{\varepsilon\theta}+I)^{-1}\overline{e_{\varepsilon\theta}}\,d\theta\,{\mathcal T}_\varepsilon{\mathcal F}_\varepsilon,
\]
where $\overline{e_{\varepsilon\theta}},$ $e_{\varepsilon\theta}$  represent the operators of multiplication by  
$\overline{e_{\varepsilon\theta}},$ $e_{\varepsilon\theta}.$
\end{proposition}
\begin{proof}[Sketch of proof]
The argument is similar to \cite{Zhikov_Pastukhova_Bloch}. Taking first solutions $(u,\nabla u)\in H^1({\mathbb R}^d, d\mu^\varepsilon)$
%L^2({\mathbb R}^d)$ 
to (\ref{whole_space_eq}) with $f\in C_0^\infty({\mathbb R}^d),$ whose both components can be shown to decay exponentially at infinity, {\it cf.} \cite[Proposition 5.3]{Zhikov_Pastukhova_Bloch}, we denote, for each such $u,$ the ``periodic amplitude'' of its scaled $\varepsilon$-Floquet transform: 
\begin{equation}
u_\theta^\varepsilon:=\overline{e_{\varepsilon\theta}}{\mathcal T}_\varepsilon{\mathcal F}_\varepsilon u=\biggl(\frac{\varepsilon^2}{2\pi}\biggr)^{d/2}\sum_{n\in Z^{d}}u(\varepsilon y+\varepsilon n)\exp\bigl(-{\rm i}(\varepsilon y+\varepsilon n)\cdot\theta\bigr).
\label{combined}
\end{equation}
Note that for any choice of the gradient $\nabla u$ (and hence for the one entering (\ref{whole_space_eq})) the expression
\[
\nabla(e_{\varepsilon\theta}u_\theta^\varepsilon)(y)=\varepsilon\biggl(\frac{\varepsilon^2}{2\pi}\biggr)^{d/2}\sum_{n\in Z^{d}}\nabla u(\varepsilon y+\varepsilon n)\exp\bigl(-{\rm i}(\varepsilon y+\varepsilon n)\cdot\theta\bigr),\qquad y\in Q,
\]
%e_{\varepsilon{\theta}}
is a gradient of $e_{\varepsilon\theta}u_\theta^\varepsilon,$ in the sense that $\bigl(e_{\varepsilon\theta}u_\theta^\varepsilon, \nabla(e_{\varepsilon\theta}u_\theta^\varepsilon)\bigr)\in H^1_{\varepsilon\theta},$ as shown by considering an appropriate sequence $\phi_n\in C_0^\infty({\mathbb R}^d)$ whose classical gradients converge to $\nabla u$ in $L^2({\mathbb R}^d, d\mu^\varepsilon).$ Therefore
%Performing Floquet trasform obtain the following problems:
\begin{equation}
\varepsilon^{-2}\int_QA\nabla(e_{\varepsilon\theta}u_\theta^\varepsilon)\cdot{\nabla(e_{\varepsilon\theta}\varphi)}\,d\mu+\int_Qe_{\varepsilon\theta}u_\theta^\varepsilon\,\overline{e_{\varepsilon\theta}\varphi}\,d\mu=\int_Qe_{\varepsilon\theta}F\,\overline{e_{\varepsilon\theta}\varphi}\,d\mu\qquad\forall\varphi\in C_\#^\infty,
%\qquad\theta\in\varepsilon^{-1}[-\pi,\pi)^d,
\label{fibre_id}
\end{equation}
where $F=\overline{e_{\varepsilon\theta}}{\mathcal T}_\varepsilon{\mathcal F}_\varepsilon f.$ The density of $f\in C_0^\infty({\mathbb R}^d)$ in 
$L^2({\mathbb R}^d, d\mu^\varepsilon)$ implies the claim.
% if the 
%\[
%e_\varkappa(y):=\exp({\rm i}\varkappa\cdot y),\qquad y\in \mathbb R^d,\qquad \varkappa\in[-\pi, \pi)^d,
%\]
\end{proof}

In what follows we study the asymptotic behaviour of the solutions $u^\varepsilon_\theta$ to the problems
\begin{equation}
\varepsilon^{-2}\overline{e_{\varepsilon\theta}}\nabla\cdot A\nabla(e_{\varepsilon\theta}u_\theta^\varepsilon)+u_\theta^\varepsilon=F,\qquad\varepsilon>0,\ \ \theta\in\varepsilon^{-1}Q',
\label{strong_form}
\end{equation}
understood in the sense of the identity (\ref{fibre_id}).

\section{Asymptotic approximation of $u^\varepsilon_\theta$}

\label{main_result}

In what follows we make the following key assumption on the measure $\mu.$

%%%Henceforth we assume that the measure $\mu$ is ergodic, {\it i.e.} whenever
%%%$\phi_n\in C_\#^\infty$ and the classical gradients $\nabla\phi_n$ converge to zero in $\bigl[L^2(Q,d\mu)\bigr]^d,$ there exists a constant $c$ such that $\phi_n\to c$ in 
%%%$L^2(Q, d\mu).$ We also assume ({\it cf.} \cite{Zhikov_Pastukhova_Bloch}) that the embedding $H^1_\#(Q,d\mu)\subset L^2(Q,d\mu)$ is compact. 
\begin{assumption}
\label{ass1}
%%Suppose that for the measure $\mu$ t
There exists $C_{\rm P}=C_{\rm P}(\mu)>0$ such that 
for all $\varkappa\in Q'$ and $\bigl(e_\varkappa u, \nabla (e_\varkappa u)\bigr)\in H_\varkappa^1$ the Poincar\'{e}-type inequality holds:
%% ({\it cf.} (\ref{Poincare_standard}))
\begin{equation}
\biggl\Vert u-\int_Qu\biggr\Vert_{L^2(Q)}\le  C_{\rm P}\bigl\Vert\nabla(e_\varkappa u)\bigr\Vert_{[L^2(Q)]^d}.
\label{Poincare}
\end{equation}
%%holds.
\end{assumption}
In Section \ref{planes_sec} we provide an example of a class of singular measures that satisfy the above assumption. This class can be extended further, see {\it e.g.} \cite{ChDO_20}.

In what follows we also assume that $A$ is a scalar matrix. The analysis of the general case is similar: the modifications required concern the condition on the mean of the unit-cell solutions defined next.
%_n\in C_\#^\infty$

%For each $\theta\in[-\pi,\pi)^d,$ 
Consider the vector $N=(N_1, N_2, ..., N_d)$ of solutions to the
%$\varkappa$-dependent 
unit cell problems\footnote{In the case of matrix-valued $A,$ the condition on the mean of the solutions $N_j,$ $j=1,2,\dots, d,$  is replaced by 
$\int_Q(A\theta\cdot\theta)N_j=0$ for $\theta\neq0,$ with no condition imposed for 
$\theta=0,$ so the mean of $N_j$ (but not its gradient) depends on $\theta.$}
\begin{equation}
-\nabla\cdot A\nabla N_j=\partial_jA,\qquad \int_Q AN_j=0,\qquad j=1,2,\dots, d.
\label{N_equation}
\end{equation}
The right-hand side of (\ref{N_equation}) is understood as an element of the space $(H^1_\#)^*$ of linear continuous functionals on $H^1_\#$: for a test function $\varphi\in C_\#^\infty$ the action of $\partial_jA$ on $\varphi$ is given by
\[
\bigl\langle\partial_jA, \varphi\bigr\rangle=\int_Q A \,\overline{\partial_j\varphi},
\]
and the action of the same functional on the whole space $H^1_\#$ is obtained by closure. In particular, for a pair ${\mathcal V}=(v, \nabla v)\in H^1_\#$ we have 
\begin{equation}
\bigl\langle\partial_jA, \mathcal V\bigr\rangle=\int_Q A\,\overline{\partial_j v}.
\label{N_functional}
\end{equation}

\begin{proposition} For each $j=1,2,\dots d,$ there exists a unique solution $N_j\in H^1_\#$ to (\ref{N_equation}).
%There is $C>0$ such that 
%\[
%\Vert N\vert_{[H^1_\#(Q)]^d}\
%\]
\end{proposition}
\begin{proof}
It follows from Assumption \ref{ass1} on the measure $\mu,$ by setting $\kappa=0$ in (\ref{Poincare}), that the following Poincar\'{e} inequality holds:
%%%\footnote{The inequality (\ref{Poincare_standard}) follows from the fact that, under our assumptions on the measure $\mu,$ the spectrum of the Laplacian on the torus is discrete and its eigenvalue zero is simple.}
\begin{equation*}
\biggl\Vert u-\int_Q u\biggr\Vert_{L^2(Q)}\le C_{\rm P}\Vert\nabla u\Vert_{[L^2(Q)]^d},\quad C_{\rm P}>0,\qquad \forall\ (u, \nabla u)\in H^1_\#.
%\label{Poincare_standard}
\end{equation*}
Therefore, the sesquilinear form
\[
\int_Q A\nabla u\cdot{\nabla v},\qquad (u, \nabla u), (v, \nabla v)\in H_{\#, 0}^1,
\]
is bounded and coercive, and hence defines an equivalent inner product in $H_{\#, 0}^1.$ Bearing in mind that (\ref{N_functional}) is a linear bounded functional on $H_{\#, 0}^1,$ we infer by the Riesz representation theorem (see {\it e.g.} \cite[p.\,32]{Birman_Solomjak}) 
%Lax-Milgram lemma (see {\it e.g.} \cite{JKO})
that for each $j=1,2,\dots d,$ the equation 
\[
-\nabla\cdot A\nabla u=\partial_jA,
%\qquad \int_Q A,\qquad j=1,2,\dots, d.
\] 
has a unique solution in $\widetilde{N}_j\in H_{\#, 0}^1,$ and therefore its arbitrary solution in $H^1_\#$ has the form $\widetilde{N}_j+a,$ $a\in{\mathbb C}.$ 
Setting 
\[
a=-\biggl(\int_Q A\biggr)^{-1}\int_Q A\widetilde{N}_j,\qquad N_j:=\widetilde{N_j}+a,
\]
concludes the proof.
%(\ref{N_equation}) has a unique solutio
\end{proof}

%In this section we prove our main result, as follows.

\begin{theorem}
\label{main_theorem}
Suppose that Assumption \ref{ass1} holds for the measure $\mu.$ Then the following estimate holds for the solutions to (\ref{strong_form}) with a constant $C>0$ independent of $\varepsilon,$ $\theta,$ $F:$
\begin{equation}
\bigl\Vert u_\theta^\varepsilon-c_\theta\bigr\Vert_{L^2(Q)}\le C\varepsilon\Vert F\Vert_{L^2(Q)},
\label{main_est}
\end{equation}
where
\begin{equation}
c_\theta=c_\theta(F):=\biggl(\theta\cdot\biggl\{\int_Q A(\nabla N+I)\biggr\}\theta+1\biggr)^{-1}\int_QF,\qquad\theta\in\varepsilon^{-1}Q'.
\label{c_condition}
\end{equation}
\end{theorem}

\begin{corollary}
\label{cor_main}
Under the conditions of the above theorem, there exists $C>0$ such that
\[
\bigl\Vert u^\varepsilon-u^0\bigr\Vert_{L^2({\mathbb R}^d, d\mu^\varepsilon)}\le C\varepsilon\Vert f\Vert_{L^2({\mathbb R}^d, d\mu^\varepsilon)}\quad\forall\, \varepsilon>0,\ f\in L^2({\mathbb R}^d, d\mu^\varepsilon),
\]
where $u^\varepsilon$ are the solutions to the original family (\ref{whole_space_eq}) and $u^0$ is the solution to the homogenised equation (\ref{u0_eq}) with 
\[
A^{\rm hom}:=\int_Q A(\nabla N+I).
\] 
\end{corollary}
%\subsection{Inner region: $\vert\theta\vert\le1$}
\begin{proof}[Proof of Corollary \ref{cor_main}]
%Throughout the proof we shall drop the superscript $\varepsilon$ in $f^\varepsilon$ for brevity. For each element of the sequence $f=f^\varepsilon\in L^2(\R^3, d\mu^\varepsilon),$ consider the $Q$-periodic function $f^\varepsilon_\theta:= \overline{e_{\varepsilon\theta}}\mathcal{T}_\varepsilon\mathcal{F}_\varepsilon f,$ {\it cf.} (\ref{u_theta^varepsilon})
Consider $f\in L^2({\mathbb R}^d, d\mu^\varepsilon)$ and denote (see (\ref{combined}))
$
f_\theta^\varepsilon:=\overline{e_{\varepsilon\theta}}{\mathcal T}_\varepsilon{\mathcal F}_\varepsilon f,
$ 
so that 
\[
\int_Qf_\theta^\varepsilon=\widehat{f}(\theta),\quad \theta\in\varepsilon^{-1}Q', \qquad{\rm where}\quad \widehat{f}(\theta):=(2\pi)^{-d/2}\int_{{\mathbb R}^d}f\overline{e_\theta}\,d\mu^\varepsilon,\quad \theta\in{\mathbb R}^d.
\]
Also, consider the solutions $u_\theta^\varepsilon$ to (\ref{strong_form}) with $F=f_\theta^\varepsilon.$ Using Proposition \ref{Floquet_resolvent} we obtain
\begin{align*}
&({\mathcal A}^\varepsilon+I)^{-1}f-({\mathcal A}^{\rm hom}+I)^{-1}f
={\mathcal F}_\varepsilon^{-1}{\mathcal T}_\varepsilon^{-1}
%\int_{\varepsilon^{-1}Q'}^\oplus 
e_{\varepsilon\theta}(\varepsilon^{-2}{\mathcal A}_{\varepsilon\theta}+I)^{-1}f_\theta^\varepsilon-({\mathcal A}^{\rm hom}+I)^{-1}f\nonumber\\[0.6em]
&={\mathcal F}_\varepsilon^{-1}{\mathcal T}_\varepsilon^{-1}e_{\varepsilon\theta}u_\theta^\varepsilon-({\mathcal A}^{\rm hom}+I)^{-1}f
\nonumber\\[0.6em]
&=\bigl\{{\mathcal F}_\varepsilon^{-1}{\mathcal T}_\varepsilon^{-1}e_{\varepsilon\theta}u_\theta^\varepsilon-{\mathcal F}_\varepsilon^{-1}{\mathcal T}_\varepsilon^{-1}e_{\varepsilon\theta}c_\theta(f_\theta^\varepsilon)\bigr\}+\bigl\{{\mathcal F}_\varepsilon^{-1}{\mathcal T}_\varepsilon^{-1}e_{\varepsilon\theta}c_\theta(f_\theta^\varepsilon)-({\mathcal A}^{\rm hom}+I)^{-1}f\bigr\},\nonumber
\end{align*}
where the operators ${\mathcal A}^\varepsilon,$ ${\mathcal A}^{\rm hom}$ are defined at the end of Section \ref{intro}. 
In view of Theorem \ref{main_theorem}, the unitary property of ${\mathcal F}_\varepsilon,$ ${\mathcal T}_\varepsilon$ and the operator of multiplication by $e_{\varepsilon\theta},$ as well as the fact that
\begin{align*}
&{\mathcal F}_\varepsilon^{-1}{\mathcal T}_\varepsilon^{-1}e_{\varepsilon\theta}(\theta\cdot A^{\rm hom}\theta+1)^{-1}\widehat{f}(\theta)-(2\pi)^{-d/2}\int_{{\mathbb R}^d}(\theta\cdot A^{\rm hom}\theta+1)^{-1}\widehat{f}(\theta)
%\biggl(\int_{{\mathbb R}^d}f(z)\exp(-2\pi{\rm i}\theta\cdot z)dz\biggr)
e_\theta\,d\theta\nonumber\\[0.3em]
%\qquad\theta\in\varepsilon^{-1}Q'
&=(2\pi)^{-d/2}\biggl\{\int_{\varepsilon^{-1}Q'}(\theta\cdot A^{\rm hom}\theta+1)^{-1}\widehat{f}(\theta)e_\theta\,d\theta
%\int_{{\mathbb R}^d}\exp(2\pi{\rm i}\theta\cdot x)(\theta\cdot A^{\rm hom}\theta+1)^{-1}\biggl(\int_Qf_\theta\biggr)d\theta\\
-\int_{{\mathbb R}^d}(\theta\cdot A^{\rm hom}\theta+1)^{-1}\widehat{f}(\theta)e_\theta\,d\theta
%\biggl(\int_{{\mathbb R}^d}f(z)\exp(-2\pi{\rm i}\theta\cdot z)dz\biggr)\exp(2\pi{\rm i}\theta\cdot x)d\theta
\biggr\}\nonumber\\[0.3em]
&=-(2\pi)^{-d/2}\int_{{\mathbb R}^d\setminus\varepsilon^{-1}Q'}(\theta\cdot A^{\rm hom}\theta+1)^{-1}\widehat{f}(\theta)e_\theta\,d\theta,\nonumber
%-\biggl(\theta\cdot\biggl\{\int_Q A(\nabla N+I)\biggr\}\theta+1\biggr)^{-1}\int_QF\biggr\}
\end{align*}
we obtain
\[
\bigl\Vert({\mathcal A}^\varepsilon+I)^{-1}f-({\mathcal A}^{\rm hom}+I)^{-1}f\bigr\Vert_{L^2({\mathbb R}^d,d\mu^\varepsilon)}\le C\varepsilon\Vert f\Vert_{L^2({\mathbb R}^d,d\mu^\varepsilon)}+\frac{\varepsilon^2}{\bigl\Vert(A^{\rm hom})^{-1}\bigr\Vert^{-1}\pi^2+\varepsilon^2}\bigl\Vert\widehat{f}\bigr\Vert_{L^2({\mathbb R}^d)},
\]
from which the claim follows.
\end{proof}

We now proceed to the proof of Theorem \ref{main_theorem}.
Motivated by formal asymptotics in powers of $\varepsilon,$ we consider the function
\begin{equation}
U_\theta^\varepsilon:=c_\theta+{\rm i}\varepsilon N_j\theta_j c_\theta+\varepsilon^2 R^\varepsilon_\theta,
\label{U_def}
\end{equation}
where $\nabla N_j,$ $j=1,2,\dots, d,$ are defined by (\ref{N_equation}), and the ``remainder'' $R^\varepsilon_\theta\in H^1_\#$ solves
\begin{align}
&-\overline{e_{\varepsilon\theta}}\nabla\cdot A\nabla(e_{\varepsilon\theta}R^\varepsilon_\theta)+\varepsilon^2\int_QR^\varepsilon_\theta=F+\varepsilon^{-2}\overline{e_{\varepsilon\theta}}\nabla\cdot A\nabla(e_{\varepsilon\theta}c_\theta)+{\rm i}\varepsilon^{-1}\overline{e_{\varepsilon\theta}}\nabla\cdot A\nabla(e_{\varepsilon\theta}N_j\theta_j)c_\theta-c_\theta
\nonumber\\[0.5em]
%\label{first_line}\\
&\equiv F+{\rm i}\nabla\cdot \bigl({\rm i} N_j\theta_jA\theta\bigr)c_\theta+{\rm i}\theta\cdot A\nabla({\rm i} N_j\theta_j)c_\theta
%\nonumber\\[0.5em]
%&
-{\rm i}\varepsilon N_j\theta_j\theta\cdot A\theta c_\theta-\theta\cdot A\theta c_\theta-c_\theta=:H^\varepsilon_\theta,
\label{R_equation}
\end{align}
where $H^\varepsilon_\theta$ is an treated as an element of the space $(H^1_\#)^*.$
Here for all $\varkappa\in Q'$ we set 
\begin{equation*}
\nabla e_\varkappa={\rm i}e_\varkappa\varkappa,\qquad \nabla(e_\varkappa N_j)=e_\varkappa({\rm i}N_j\varkappa+\nabla N_j),\qquad1,2,\dots, d.
%\label{grad_def}
\end{equation*}
The second equality in (\ref{R_equation}) is verified by taking $\varphi\in C^\infty_\#,$ noticing that
\begin{align*}
&\Bigl\langle\varepsilon^{-2}\overline{e_{\varepsilon\theta}}\nabla\cdot A\nabla e_{\varepsilon\theta}+{\rm i}\varepsilon^{-1}\overline{e_{\varepsilon\theta}}\nabla\cdot A\nabla(e_{\varepsilon\theta}N_j\theta_j), \varphi\Bigr\rangle
\\[0.5em]
&=
-\int_Q\Bigl(\varepsilon^{-2}A{\rm i}e_{\varepsilon\theta}\varepsilon\theta\cdot{\nabla(e_{\varepsilon\theta}\varphi)}
+{\rm i}\varepsilon^{-1}
Ae_{\varepsilon\theta}\theta_j({\rm i}N_j\varepsilon\theta+\nabla N_j)\cdot{\nabla(e_{\varepsilon\theta}\varphi)}\Bigr)\\[0.5em]
&=-\int_Q\Bigl(\varepsilon^{-2}A{\rm i}e_{\varepsilon\theta}\varepsilon\theta\cdot{({\rm i}e_{\varepsilon\theta}\phi\varepsilon\theta+e_{\varepsilon\theta}\nabla\varphi)}+{\rm i}\varepsilon^{-1}
Ae_{\varepsilon\theta}\theta_j({\rm i}N_j\varepsilon\theta+\nabla N_j)\cdot{({\rm i}e_{\varepsilon\theta}\phi\varepsilon\theta+e_{\varepsilon\theta}\nabla\varphi)}
%\overline{\nabla(e_{\varepsilon\theta}\phi)}
\Bigr),
\end{align*}
and finally using (\ref{N_equation}). Note that for $c_\theta$ defined by (\ref{c_condition}), the condition $\langle H^\varepsilon_\theta, 1\rangle=0$ holds, and in the case $\theta=0$ the average over $Q$ of the solution $R^\varepsilon_\theta$ to (\ref{R_equation}) vanishes. 

%Proposition \ref{Poincare_estimate} and 
The estimate (\ref{Poincare}) and the classical Riesz representation theorem \cite[p.\,32]{Birman_Solomjak} imply that for each $\varepsilon>0,$ $\theta\in\varepsilon^{-1}Q',$  there exists a unique solution $R^\varepsilon_\theta\in H^1_\#$ to the problem (\ref{R_equation}).

\section{Discussion of the validity of (\ref{Poincare}) for some singular measures}

\label{planes_sec}

%%Note that that for all $\bigl(e_\varkappa u, \nabla (e_\varkappa u)\bigr)\in H_\varkappa^1$ one has $\nabla(e_\varkappa u)=e_\varkappa({\rm i}u\varkappa+\nabla u)$ for some $(u,\nabla u)\in H_\#^1.$ Therefore, in order to prove (\ref{Poincare}) it suffices to minimise
%%\begin{equation*}
%%\biggl(\int_Q\vert u\vert^2\,d\mu\biggr)^{-1}\int_Q\bigl\vert{\rm i}u\varkappa+\nabla u\bigr\vert^2\,d\mu,\qquad u\in H_\#^1,\ \int_Qu=0,\ u\neq0,
%%\end{equation*}
%%and then take the infimum over $\varkappa\in[-\pi,\pi)^d.$

%%For the case when $\mu$ is the Lebesgue measure, one has $(C_{\rm P}(\mu))^{-2}=\min_\varkappa(\varkappa+2\pi)^2=\pi^2.$ 
%%The same value of $C_{\rm P}$ applies to the case of the linear measure supported by the ``square grid'', see \cite[Section 9.2]{Zhikov_Pastukhova_Bloch}. 

Consider a finite set $\{{\mathcal P}_j\}_{j=1}^N$ of hyperplanes of dimension $d$ or smaller each of which 
%%${\mathcal P}_j$
is parallel some of the Euclidean coordinate axes in ${\mathbb R}^d$ and orthogonal to the complementary coordinate axes
%%for all $j$
%%%and ${\mathcal P}_j$ is not a subset of ${\mathcal P}_k$ for all $j, k.$
and such that $(\cup_{j=1}^N{\mathcal P}_j)\cap Q$ is non-empty and connected.

 Define the measure $\mu$ on $Q$ by the formula 
\[
\mu(B)=\Bigl(\sum_{j=1}^N\vert {\mathcal P}_j\cap Q\vert_j\Bigr)^{-1}\sum_{j=1}^N\vert {\mathcal P}_j\cap B\vert_j\ \ {\rm for\ all\ Borel\ } B\subset Q.
\]  
where $\vert\cdot\vert_j$ represents the $d_j$-dimensional Lebesgue measure, $d_j={\rm dim}({\mathcal P}_j).$ 

%%%Then the assumptions of Theorem \ref{main_theorem} hold with $C_{\rm P}(\mu)=1/\pi^2.$ 
%%The above observation can be generalised to the case of a measure $\mu$ supported by the union $\cup_j {\mathcal P}_j$ of any finite set of hyperplanes of lower codimension, their Borel subsets of the same dimension, and the subsets of $Q,$ so that $\mu\vert_{{\mathcal P}_j}$ is the restriction of a lower-dimensional Lebesgue measure to ${\mathcal P}_j.$ 
%{\color{red} So far this only works under the assumption that the boundary of}
%%Indeed, in the case when ${\rm supp}(\mu)$ in $Q$ is a lower-dimensional unit cube (which can be assumed parallel to the coordinate axis, via a rotation), the estimate for the Rayleigh quotient in (\ref{Rayleigh_inf}) is similar to the case of $d$-dimensional Lebesgue measure. 

For each $j\in\{1,\dots, N\}$ consider the measure $\mu_j$ defined by
\[
\mu_j(B):=\vert {\mathcal P}_j\cap Q\vert_j^{-1}\vert {\mathcal P}_j\cap B\vert_j\ \ {\rm for\ all\ Borel\ } B\subset Q,
\]
%For every fixed $j,$ 

%%%Furthermore, denote by $\widetilde{\mathcal P}_j$ the linear subspace parallel to ${\mathcal P}_j$ and passing through zero in ${\mathbb R}^d,$
 %%%

\subsection{Poincar\'{e} inequality for a single hyperplane}

In this section, we fix $j\in\{1,\dots, N\}$ and assume, without loss of generality,
%, as in Section \ref{curl0}, 
%assume 
that the plane 
${\mathcal P}_j$ passes through zero. We denote by $Q_j$ the $d_j$-dimensional cross-section of $Q$ by ${\mathcal P}_j,$
%%% ({\it i.e.} the cell $[0,1)^{d_j}$), 
 by $\varkappa_j^{||}$ the vector of those components of the quasimomentum $\varkappa$ that correspond to the selection of the coordinates entering ${\mathcal P}_j$ considered as a subspace of ${\mathbb R}^d,$ and by $\varkappa_j^\perp$ the vector of those components of $\varkappa$ that do not enter $\varkappa_j^{||}.$ 
%%coincides with $\widetilde{\mathcal P}_j$ 
%%%%and is orthogonal to the $x_3$-direction. 
%%%%For a function $\phi\in C_\#^\infty,$ 
%%%with zero mean over $Q,$ at each point $x\in Q,$ 
%$\int_Q\phi=0,$ 
%$\phi\neq0,$ 
%%we decompose the (classical) gradient $\nabla\phi(x)$ into the orthogonal sum of its projection 
%%%%we denote by $\widetilde{\nabla}\phi(x)\in{\mathbb R}^2,$ $x\in Q,$ the (pointwise)  projection of its gradient onto the $(x_1, x_2)$-plane.

 For a function $\phi\in C_\#^\infty,$ %%with zero mean over $Q,$
 at each point $x\in Q,$ 
%$\int_Q\phi=0,$ 
%$\phi\neq0,$ 
we decompose the (classical) gradient $\nabla\phi(x)$ into the orthogonal sum of its projection $\nabla_j^{||}\phi(x)$ onto $\widetilde{\mathcal P}_j$  and its projection $\nabla_j^\perp\phi(x)$ onto the orthogonal complement of $\widetilde{\mathcal P}_j.$ We treat $\nabla_j^{||}\phi(x)$ and $\nabla_j^\perp\phi(x)$ as elements of ${\mathbb R}^{d_j}$ and ${\mathbb R}^{d-d_j},$ respectively. Clearly, for each $\varkappa\in Q',$ 
%%%denoting by $\varkappa_j^\perp$ the vector of those components of $\varkappa$ that do not enter $\varkappa_j^{||},$
%%$\varkappa=(\widtilde{\varkappa}_j, \varkappa_j^{||})$  
one has, pointwise in $Q,$
\begin{equation}
\bigl\vert\nabla(e_\kappa\phi)\bigr\vert=\vert{\rm i}\phi\varkappa+\nabla\phi\bigr\vert^2=\vert{\rm i}\phi\varkappa_j^{||}+\nabla^{||}\phi\bigr\vert^2+\vert{\rm i}\phi\varkappa_j^\perp+\nabla^\perp\phi\bigr\vert^2\ge \vert{\rm i}\phi\varkappa_j^{||}+\nabla^{||}\phi\bigr\vert^2,
\label{Eineq}
\end{equation}
where the norms are considered in appropriate Euclidean spaces.

Next, we write 
\[
\phi(\widetilde{x})-\int_Q\phi d\mu_j=\sum_{l\in{\mathbb Z}^{d_j}\setminus\{0\}}c_l\exp(2\pi{\rm i} l\cdot\widetilde{x}), \quad \widetilde{x}\in Q_j,\qquad c_l\in{\mathbb C},\ l\in{\mathbb Z}^{d_j},
\]
and notice that for all $j$ one has, assuming $\phi$ is non-constant on ${\mathcal P}_j\cap Q,$
\begin{align*}
%\[
&\biggl(\int_Q\biggl\vert\phi-\int_Q\phi d\mu_j\biggr\vert^2d\mu_j\biggr)^{-1}\int_Q\bigl\vert{\rm i}\phi\varkappa_j^{||}+\nabla^{||}\phi\bigr\vert^2d\mu_j\\[0.5em]
&=\biggl(\sum_{l, m\in{\mathbb Z}^{d_j}\setminus\{0\}}\alpha_{lm}c_l\overline{c_m}\biggr)^{-1}\biggl(\sum_{l, m\in{\mathbb Z}^{d_j}\setminus\{0\}}\alpha_{lm}c_l\overline{c_m}(\varkappa_j^{||}+2\pi l)\cdot(\varkappa_j^{||}+2\pi m)\biggr),
%%\ge\pi^2,
\end{align*}
where
%\qquad 
\[
\alpha_{lm}:=\int_{Q_j}\exp\bigl(2\pi{\rm i}(l-m)\cdot\widetilde{x}\bigr)d\mu_j(\widetilde{x})=\left\{\begin{array}{ll}1, \ \ l=m,\\[0.4em]
0\ \ \ {\rm otherwise}.\end{array}\right.
\]
It follows that 
\begin{equation*}
\begin{aligned}
%\begin{align*}
%&
\biggl(\int_Q\biggl\vert\phi-\int_Q\phi d\mu_j\biggr\vert^2d\mu_j\biggr)^{-1}&\int_Q\bigl\vert{\rm i}\phi\varkappa_j^{||}+\nabla^{||}\phi\bigr\vert^2d\mu_j
\\[0.5em]
&=\biggl(\sum_{l\in{\mathbb Z}^{d_j}\setminus\{0\}}|c_l|^2\biggr)^{-1}\biggl(\sum_{l, m\in{\mathbb Z}^{d_j}\setminus\{0\}}
|c_l|^2|\varkappa_j^{||}+2\pi l|^2\biggr)\ge\pi^2,
%\end{align*}
\end{aligned}
\end{equation*}
or equivalently
\begin{equation}
\int_Q\biggl\vert\phi-\int_Q\phi d\mu_j\biggr\vert^2d\mu_j\le\pi^{-2}\int_Q\bigl\vert{\rm i}\phi\varkappa_j^{||}+\nabla^{||}\phi\bigr\vert^2d\mu_j
\label{lem_scal}
\end{equation}
If the function $\phi$ is constant on ${\mathcal P}_j\cap Q,$ the inequality (\ref{lem_scal}) is satisfied trivially.

%%%and therefore, due to (\ref{Eineq}),
%%%\[
%%%\biggl(\int_Q\biggl\vert\phi-\int_Q\phi d\mu_j\biggr\vert^2d\mu_j\biggr)^{-1}\int_Q\bigl\vert{\rm i}\phi\varkappa+\nabla\phi\bigr\vert^2d\mu_j\ge\biggl(\int_Q\biggl\vert\phi-\int_Q\phi d\mu_j\biggr\vert^2d\mu_j\biggr)^{-1}\int_Q\bigl\vert{\rm i}\phi\varkappa_j^{||}+\nabla^{||}\phi\bigr\vert^2d\mu_j\ge\pi^2.
%%%\]

\subsection{Connectivity argument}
\label{connect}

For the measure $\mu=\sum_{j=1}^N\mu_j$ and $\phi\in C^\infty_\#,$ we denote by 
$\nabla^{||}(e_\kappa\phi)$ the tangential gradient of $\phi$ at points of 
${\rm supp}(\mu),$ {\it i.e.} the orthogonal projection of $\nabla(e_\kappa\phi)$ onto ${\rm supp}(\mu).$  

Suppose that for $j,k\in\{1,\dots, N\}$ the hyperplanes ${\mathcal P}_j$ and ${\mathcal P}_k$  intersect and fix a point $\alpha_{jk}\in{\mathcal P}_j\cap{\mathcal P}_k\cap Q.$ For any $\kappa\in Q',$ any function 
$\phi\in C^\infty_\#,$ and all $x\in{\mathcal P}_j\cap Q$, $y\in{\mathcal P}_k\cap Q,$ one has
\begin{equation}
\begin{aligned}
e_\kappa(x)\phi(x)-e_\kappa(y)\phi(y)&=\int_{\alpha_{jk}}^x\nabla(e_\kappa\phi)\bigl(\alpha_{jk}+t(x-\alpha_{jk})\bigr)dt\cdot(x-\alpha_{jk})\\[0.5em]
&-\int_{\alpha_{jk}}^y\nabla(e_\kappa\phi)\bigl(\alpha_{jk}+t(y-\alpha_{jk})\bigr)dt\cdot(y-\alpha_{jk}).
\end{aligned}
\label{both_sides}
\end{equation} 
Multiplying both sides of (\ref{both_sides}) by $e_\kappa(y)^{-1}=e_\kappa(-y)$ and integrating over $y\in Q$ with respect to the measure $\mu_k$ (recalling that ${\rm supp}(\mu_k)={\mathcal P}_k\cap Q$) yields
%% the existence of $C_{jk}>0$ such that
\begin{equation}
e_\kappa(x)\phi(x)\int_Qe_\kappa^{-1}d\mu_k-\int_Q\phi d\mu_k\le 
%%C_{jk}
\sqrt{2}\Bigl(\bigl\Vert\nabla^{||}(e_\kappa\phi)\bigr\Vert_{L^2(Q,d\mu_j)}+\bigl\Vert\nabla^{||}(e_\kappa\phi)\bigr\Vert_{L^2(Q,d\mu_k)}\Bigr)\qquad\forall x\in{\mathcal P}_j\cap Q.
\label{interm_phi}
\end{equation}
Furthermore, multiplying both sides of (\ref{interm_phi}) by $e_\kappa(x)^{-1}$ and integrating over $x\in Q$ with respect to the measure $\mu_j$ yields
\begin{equation*}
\int_Q\phi d\mu_j-\int_Q\phi d\mu_k\le 
%%C_{jk}
\sqrt{2}\Bigl(\bigl\Vert\nabla^{||}(e_\kappa\phi)\bigr\Vert_{L^2(Q,d\mu_j)}+\bigl\Vert\nabla^{||}(e_\kappa\phi)\bigr\Vert_{L^2(Q,d\mu_k)}\Bigr)\le\sqrt{2}\Vert\nabla^{||}(e_\kappa\phi)\bigr\Vert_{L^2(Q,d\mu)}.
%\label{jk_neighbour}
\end{equation*}
By interchanging $k$ and $j$ if necessary, we thus obtain 
\[
\biggl\vert\int_Q\phi d\mu_j-\int_Q\phi d\mu_k\biggr\vert\le \sqrt{2}\bigl\Vert\nabla^{||}(e_\kappa\phi)\bigr\Vert_{L^2(Q,d\mu)}.
\]

Next, notice that since $(\cup_{j=1}^N{\mathcal P}_j)\cap Q$ is connected by assumption, for each pair of planes in the union there is a ``path" from one plane to the other involving at most $N$ planes, such that any ``adjacent" planes in the path intersect. It follows that for all pairs $j, k$ the bound 
%%(\ref{jk_neighbour})
\begin{equation}
\biggl\vert\int_Q\phi d\mu_j-\int_Q\phi d\mu_k\biggr\vert\le \sqrt{2}N\bigl\Vert\nabla^{||}(e_\kappa\phi)\bigr\Vert_{L^2(Q,d\mu)}.
\label{anyjk}
\end{equation}
 holds.
%%$ with come $C_{jk}>0.$ 
%${\mathcal P}_j$

Finally, using (\ref{anyjk}) and standard arithmetic inequalities, we obtain
%\begin{equation*}
\begin{align*}
\int_Q\biggl\vert\phi-\int_Q\phi\biggr\vert^2d\mu&=\sum_{j=1}^N\int_Q\biggl\vert\phi-\int_Q\phi\biggr\vert^2d\mu_j=\sum_{j=1}^N\int_Q\biggl\vert\phi-\sum_{k=1}^NN^{-1}
%{\mathfrak f}_k
\int_Q\phi d\mu_k\biggr\vert^2d\mu_j\nonumber\\[0.4em]
&=\sum_{j=1}^N\int_Q\biggl\vert\sum_{k=1}^N
%{\mathfrak f}_k
N^{-1}\biggl(\phi-\int_Q\phi d\mu_k\biggr)\biggr\vert^2d\mu_j\le \sum_{j=1}^N\sum_{k=1}^N
%{\mathfrak f}_k
N^{-1}\int_Q\biggl\vert\biggl(\phi-\int_Q\phi d\mu_k\biggr)\biggr\vert^2d\mu_j\nonumber\\[0.4em]
&=\sum_{j=1}^N\sum_{k=1}^N
%{\mathfrak f}_k
N^{-1}
\int_Q\biggl\vert\phi-\int_Q\phi d\mu_j+
%\sum_{\substack{(s,t){\rm\,in\,path}\\ {\rm connecting\,}j, k}}\biggl(\int_Q\phi d\mu_s-\int_Q\phi d\mu_t\biggr)
\biggl(\int_Q\phi d\mu_j-\int_Q\phi d\mu_k\biggr)
\biggr\vert^2d\mu_j\nonumber\\[0.4em]
&\le 2\sum_{j=1}^N\sum_{k=1}^N
%{\mathfrak f}_k
N^{-1}\biggl\{\int_Q\biggl\vert\phi-\int_Q\phi d\mu_j
\biggr\vert^2d\mu_j+2N^2\Vert\nabla^{||}(e_\kappa\phi)\bigr\Vert^2_{L^2(Q,d\mu)}\biggr\}\nonumber\\[0.4em]
&=2\sum_{j=1}^N\biggl\{\int_Q\biggl\vert\phi-\int_Q\phi d\mu_j
\biggr\vert^2d\mu_j+2N^2\Vert\nabla^{||}(e_\kappa\phi)\bigr\Vert^2_{L^2(Q,d\mu)}\biggr\}\nonumber\\[0.4em]
&\le2\sum_{j=1}^N\biggl\{\pi^{-2}\Vert\nabla^{||}(e_\kappa\phi)\bigr\Vert^2_{L^2(Q,d\mu_j)}+2N^2\Vert\nabla^{||}(e_\kappa\phi)\bigr\Vert^2_{L^2(Q,d\mu)}\biggr\}\nonumber\\[0.4em]
&\le 2N\bigl(\pi^{-2}+2N^2\bigr)\Vert\nabla^{||}(e_\kappa\phi)\bigr\Vert^2_{L^2(Q,d\mu)}.
%\label{pre_final_P}
%%\sum_{(s,t){\rm\  in\ path}}
%C_{st}
%%\Bigl(\bigl\Vert\nabla^{||}(e_\kappa\phi)\bigr\Vert_{L^2(Q,d\mu_j)}^2+\bigl\Vert\nabla^{||}%%(e_\kappa\phi)\bigr\Vert_{L^2(Q,d\mu_k)}^2\Bigr)\biggr\}
\end{align*}
%\end{equatiion*}
Combining the above bound
%%, the estimate (\ref{lem_scal}), 
with (\ref{Eineq}), where we notice that for each $j=1,\dots, N,$ on ${\rm supp}(\mu_j)$ one has 
\[
\nabla^{||}(e_\kappa\phi)=e_\kappa({\rm i}\phi\varkappa_j^{||}+\nabla^{||}\phi),
\]
%% the result of Proposition \ref{curl_grad_prop} 
%%and the bound (\ref{penult}), 
%%applied for each $j=1,\dots, N,$ 
we obtain 
\begin{equation}
\int_Q\biggl\vert\phi-\int_Q\phi\biggr\vert^2d\mu\le C_{\rm P}\bigl\|\nabla(e_\kappa \phi)\bigr\|^2_{L^2(Q,d\mu)},
\label{phi_fin}
\end{equation}
with 
%%%%%the constant $C_{\rm P}$ given by
\begin{equation}
C_{\rm P}=2N\bigl(\pi^{-2}+2N^2\bigr).
\label{CP}
\end{equation}

%%%%%According to the result of Section \ref{smooth_approx}, the pair $({\mathfrak u}, \curl(e_\kappa{\mathfrak u}))$ is approximated by functions $\phi_n\in[C^\infty_\#]^3$ satisfying the conditions of Proposition \ref{curl_grad_prop}, {\it i.e.} such that ({\it cf.} (\ref{weak_div_cond}))
%%%%%\[
%%%%%\int_Qe_\kappa\phi_n\cdot {\nabla(e_\kappa\psi)}\,d\mu=0\qquad \forall\psi\in C_\#^\infty
%%%%%\]
%%%%and $\curl(e_\kappa\phi_n)$ is pointwise orthogonal to ${\rm supp}(\mu),$
%%%%\begin{equation}
%%%%\overline{e}_\kappa{\rm div}\,(e_\kappa \phi_n)=0
%%%%\label{kappa_sol}
%%%%\end{equation}
%%%%in the sense of usual derivatives and, {\color{red} in addition, the identity (\ref{sol_weak}) holds,}
%%\[
%%\overline{e}_\kappa{\rm div}\,(e_\kappa \phi_n)=0
%%\]
%%in the sense of measure $\mu,$ see (\ref)
%%%%%where the approximation is understood in the sense that ({\it cf.} (\ref{phin_approx}))
%%%%%\[
%%%%%\bigl(e_\kappa \phi_n, \curl(e_\kappa\phi_n)\bigr)\to\bigl({\mathfrak u}, \curl(e_\kappa{\mathfrak u})\bigr)\ \ {\rm in}\ \ L^2(Q, d\mu) \oplus L^2(Q, d\mu).
%%%%%\]

%%Approximating given vector $u\in L^2(Q,d\mu)$ by such smooth functions, 

Finally, approximating an arbitrary $\bigl(e_\varkappa u, \nabla (e_\varkappa u)\bigr)\in H_\varkappa^1$ by pairs $\bigl(e_\varkappa\phi, \nabla (e_\varkappa\phi)\bigr),$ $\phi\in C^\infty_\#,$ in line with
%$u\in H^1_\#$ 
 Definition \ref{Sobolev_definition},
%%Writing the bound (\ref{phi_fin}) with $\phi=\phi_n,$ where $\{\phi_n\}\subset[C^\infty_\#$ is the approximating sequence for $u$ as described above, 
and passing to the limit as $n\to\infty$ in the bound (\ref{phi_fin})
%%%\begin{equation}
%%%\int_Q\biggl\vert u-\int_Qu\biggr\vert^2d\mu\le C_{\rm P}
%%%\bigl\|\curl(e_\kappa \phi)\bigr\|^2_{L^2(Q,d\mu)}
%%2N\bigl(\pi^{-2}+2N^2\bigr)
%%\Vert\nabla^{||}(e_\kappa{\mathfrak u})\bigr\Vert^2_{L^2(Q,d\mu)}
%%%\label{penult}
%%%\end{equation}
%%{\it i.e.}
yields the inequality (\ref{Poincare}), with $C_{\rm P}$ given by (\ref{CP}).

\section{Estimate for the ``remainder'' $\varepsilon^2R^\varepsilon_\theta$}
\label{R_est_section}
%+z^\varepsilon_\theta$}

%%\begin{lemma}
%%There exists $C>0$ such that 
%%\begin{equation}
%%\bigl\Vert H^\varepsilon_\theta\bigr\Vert_{(H^1_\#)^*}\le C\Vert F\Vert_{L^2(Q)}\qquad\forall\,\varepsilon>0,\ \theta\in\varepsilon^{-1}Q'.
%[-\pi/\varepsilon,\pi/\varepsilon)^d.
%\ \ \vert\theta\vert\ge 1.
%\Bigl\vert\bigl\langle H^\varepsilon_\theta, 1\bigr\rangle\Bigr\vert\le C\vert\varkappa\vert\Vert F\Vert_{L^2(Q)}.
%%\label{Rprob_est}
%%\end{equation}
%{\color{red} Remove $e_{\varepsilon\theta}$ and its conjugate?}
%%\end{lemma}

%%\begin{proof} By direct calculation, using the definitions of $N,$ $c_\theta$ and boundedness of $A.$
%%\end{proof}

\begin{theorem}
\label{R_est_lemma}
Suppose that $\theta\neq 0,$ $\varepsilon>0.$ For the solution $R^\varepsilon_\theta$ to the problem (\ref{R_equation}) the following estimates hold with $C>0:$
%\[
%-\overline{e_{\varepsilon\theta}}\nabla\cdot A\nabla(e_{\varepsilon\theta}R^\varepsilon_\theta)=\widetilde{H}^\varepsilon_\theta=H^\varepsilon_\theta,
%\]
%where $H^\varepsilon_\theta\in H^{-1}_\#,$ the following estimate holds:
\begin{equation}
\biggl\Vert R^\varepsilon_\theta-\int_QR^\varepsilon_\theta\biggr\Vert_{L^2(Q)}\le C\Vert F\Vert_{L^2(Q)},\qquad \biggl\vert\int_QR^\varepsilon_\theta\biggr\vert\le
%\left\{\begin{array}{ll}
C\varepsilon^{-1}\Vert F\Vert_{L^2(Q)}.
%,\ \ \theta\neq 0.
%\\[0.4em]C\Vert F\Vert_{L^2(Q)},\ \ \theta=0.\end{array}\right.
\label{target_est}
\end{equation}
\end{theorem}

\begin{proof} 

Consider a sequence of functions $\phi_n\in C_\#^\infty$ that converges in $L^2(Q)$ to $R^\varepsilon_\theta,$ such that 
\[
\nabla(e_{\varepsilon\theta}\phi_n)\stackrel{[L^2(Q)]^d}{\longrightarrow}\nabla\bigl(e_{\varepsilon\theta}R^\varepsilon_\theta\bigr),
\]
and, equivalently,
\[
\nabla\biggl[e_{\varepsilon\theta}\biggl(\phi_n-\int_QR^\varepsilon_\theta\biggr)\biggr]\stackrel{[L^2(Q)]^d}{\longrightarrow}\nabla\biggl[e_{\varepsilon\theta}\biggl(R^\varepsilon_\theta-\int_QR^\varepsilon_\theta\biggr)\biggr].
\]
%and $R^\varepsilon_\theta-\int_QR^\varepsilon_\theta,$ respectively.
%(Recall that the gradient $\nabla e_{\varepsilon\theta}$ is fixed by (\ref{nabla_e_def}).)
 It follows from (\ref{R_equation}) that
\[
\int_QA\nabla(e_{\varepsilon\theta}R^\varepsilon_\theta)\cdot {\nabla(e_{\varepsilon\theta}\phi_n)}+\varepsilon^2\int_QR^\varepsilon_\theta\overline{\int_Q\phi_n}=\bigl\langle H^\varepsilon_\theta, 1\bigr\rangle\overline{\int_QR^\varepsilon_\theta}+\biggl\langle H^\varepsilon_\theta, \phi_n-\int_QR^\varepsilon_\theta\biggr\rangle.
\]
Furthermore, using the fact that $\langle H^\varepsilon_\theta, 1\rangle=0$ and the formula
\[
\nabla\phi_n=\overline{e_{\varepsilon\theta}}
\Biggl\{\nabla\biggl[e_{\varepsilon\theta}\biggl(\phi_n-\int_QR^\varepsilon_\theta\biggr)\biggr]-\biggl(\phi_n-\int_QR^\varepsilon_\theta\biggr)\nabla e_{\varepsilon\theta}\Biggr\},
\]
where all the gradients are understood in the classical sense, we obtain
\begin{align*}
\int_QA\nabla(e_{\varepsilon\theta}R^\varepsilon_\theta)&\cdot {\nabla(e_{\varepsilon\theta}\phi_n)}+\varepsilon^2\int_QR^\varepsilon_\theta\overline{\int_Q\phi_n}=
\int_Q\bigl(F+{\rm i}\theta\cdot A\nabla({\rm i} N_j\theta_j)c_\theta
\nonumber\\[0.5em]
&
-{\rm i}\varepsilon N_j\theta_j\theta\cdot A\theta c_\theta-\theta\cdot A\theta c_\theta-c_\theta\bigr)\overline{\biggl(\phi_n-\int_QR^\varepsilon_\theta\biggr)}\nonumber\\[0.5em]
&+c_\theta\int_Qe_{\varepsilon\theta}N_j\theta_jA\theta\cdot{\Biggl\{\nabla\biggl[e_{\varepsilon\theta}\biggl(\phi_n-\int_QR^\varepsilon_\theta\biggr)\biggr]-\biggl(\phi_n-\int_QR^\varepsilon_\theta\biggr)\nabla e_{\varepsilon\theta}\Biggr\}}.
\end{align*}
Passing to the limit as $n\to\infty$ yields
\begin{align}
&\int_QA\nabla(e_{\varepsilon\theta}R^\varepsilon_\theta)\cdot{\nabla(e_{\varepsilon\theta}R^\varepsilon_\theta)}+\varepsilon^2\biggl\vert\int_QR^\varepsilon_\theta\biggr\vert^2=
\int_Q\biggl[F-c_\theta\Bigl(\theta\cdot A\nabla(N_j\theta_j)+({\rm i}\varepsilon N_j\theta_j+1)\theta\cdot A\theta
\nonumber\\[0.5em]
%-\theta\cdot A\theta
&+e_{\varepsilon\theta}N_j\theta_j\theta\cdot{A\nabla e_{\varepsilon\theta}}+1\Bigr)\biggr]\overline{\biggl(R^\varepsilon_\theta-\int_QR^\varepsilon_\theta\biggr)}
%\nonumber\\[0.5em]
%&
+c_\theta\int_Q e_{\varepsilon\theta}N_j\theta_j\theta\cdot{A\nabla\biggl\{e_{\varepsilon\theta}\biggl(R^\varepsilon_\theta-\int_QR^\varepsilon_\theta\biggr)\biggr\}}.
%-c_\theta\int_Qe_{\varepsilon\theta}N_j\theta_jA\theta\cdot{\biggl(R^\varepsilon_\theta-\int_QR^\varepsilon_\theta\biggr)\nabla e_{\varepsilon\theta}}
%\nonumber\\[0.5em]
%&
%%-c_\theta\int_QR^\varepsilon_\theta\int_Qe_{\varepsilon\theta}N_j\theta_j\theta\cdot{A\nabla e_{\varepsilon\theta}}
\label{penultimate}
%-c_\theta)\overline{R^\varepsilon_\theta}
\end{align}

Consider the solution $\Phi^\varepsilon_\theta\in H^1_\#$ to the problem
\begin{equation}
-\overline{e_{\varepsilon\theta}}\nabla\cdot A\nabla(e_{\varepsilon\theta}\Phi^\varepsilon_\theta)+\varepsilon^2\int_Q\Phi^\varepsilon_\theta
=-\overline{e_{\varepsilon\theta}}\nabla\cdot\bigl(e_{\varepsilon\theta}N_j\theta_jA\theta\bigr)c_\theta,
\label{Phi_eq0}
\end{equation}
so that for the last term in (\ref{penultimate}) we obtain 
\begin{equation}
c_\theta\int_Q e_{\varepsilon\theta}N_j\theta_j\theta\cdot{A\nabla\biggl\{e_{\varepsilon\theta}\biggl(R^\varepsilon_\theta-\int_QR^\varepsilon_\theta\biggr)\biggr\}}=\int_QA\nabla(e_{\varepsilon\theta}\Phi^\varepsilon_\theta)\cdot{\nabla\biggl\{e_{\varepsilon\theta}\biggl(R^\varepsilon_\theta-\int_QR^\varepsilon_\theta\biggr)\biggr\}}.
\label{interm}
\end{equation}
In what follows we use the uniform estimate 
\begin{equation}
\bigl\Vert\sqrt{A}\nabla(e_{\varepsilon\theta}\Phi^\varepsilon_\theta)\bigr\Vert_{[L^2(Q)]^d}\le C\Vert F\Vert_{L^2(Q)},
\label{gradPhi_est}
\end{equation}
which is obtained by using $\Phi^\varepsilon_\theta$ as a test function in the integral formulation of (\ref{Phi_eq0}).

We would like to rewrite the expression on the right-hand side of (\ref{interm}) using $\Phi^\varepsilon_\theta$ as a test function in the integral identity for (\ref{R_equation}). Recall that the gradient of an arbitrary function in $H^1_\#,$ for a general measure $\mu,$ is not defined in a unique way.
 %is not defined in a unique way,\footnote{Recall that the solution to (\ref{Phi_eq0}) is the pair $\bigl(e_{\varepsilon\theta}\Phi^\varepsilon_\theta, \nabla(e_{\varepsilon\theta}\Phi^\varepsilon_\theta)\bigr),$ which {\it is} defined uniquely.} 
However, for the solution $\Phi^\varepsilon_\theta$ to (\ref{Phi_eq0}) there exists a natural choice of the gradient $\nabla\Phi^\varepsilon_\theta,$ dictated by (\ref{Phi_eq0}).
 % the following statement holds.
%%\begin{proposition}
%%jkjkk
%%\end{proposition}
%%\begin{proof}
%%By the definition of the function $R^\varepsilon_\theta$ as the solution to (\ref{R_equation}), for all functions $\phi\in C_\#^\infty$ one has 
%%\[
%%\int_QA\nabla(e_{\varepsilon\theta}R^\varepsilon_\theta)\cdot{\nabla(e_{\varepsilon\theta}\phi)}+\varepsilon^2\int_QR^\varepsilon_\theta\overline{\int_Q\phi}=\bigl\langle H^\varepsilon_\theta,\phi\bigr\rangle
%%\]
%%\end{proof}
Indeed, consider sequences $\phi_n,$ $\psi_n\in C^\infty_\#$ converging to $\Phi^\varepsilon_\theta$ in $L^2(Q)$ so that
\[
\nabla(e_{\varepsilon\theta}\phi_n)\stackrel{[L^2(Q)]^d}{\longrightarrow}\nabla\bigl(e_{\varepsilon\theta}\Phi^\varepsilon_\theta\bigr),\qquad\nabla(e_{\varepsilon\theta}\psi_n)\stackrel{[L^2(Q)]^d}{\longrightarrow}\nabla\bigl(e_{\varepsilon\theta}\Phi^\varepsilon_\theta\bigr).
\]
Clearly, the difference $\nabla(e_{\varepsilon\theta}\phi_n)-\nabla(e_{\varepsilon\theta}\psi_n)$ converges to zero, and hence so does $\nabla\phi_n-\nabla\psi_n.$ In what follows we denote by $\nabla\Phi^\varepsilon_\theta$ the common $L^2$-limit of gradients $\nabla\phi_n$ for sequences $\phi_n\in C^\infty_\#$ with the above properties.
Passing to the limit, as $n\to\infty,$ in the identity 
$\nabla\phi_n=\overline{e_{\varepsilon\theta}}\bigl(\nabla(e_{\varepsilon\theta}\phi_n)-{\rm i}\varepsilon\phi_n\theta\bigr),$
we obtain the formula
\begin{equation}
\nabla\Phi^\varepsilon_\theta=\overline{e_{\varepsilon\theta}}\bigl(\nabla(e_{\varepsilon\theta}\Phi^\varepsilon_\theta)-{\rm i}\varepsilon\Phi^\varepsilon_\theta\theta\bigr).
\label{grad_Phi}
\end{equation}

The unique choice of $\nabla\Phi^\varepsilon_\theta,$ as above, allows us to write
\[
\int_QA\nabla(e_{\varepsilon\theta}R^\varepsilon_\theta)\cdot{\nabla(e_{\varepsilon\theta}\Phi^\varepsilon_\theta)}+\varepsilon^2\int_QR^\varepsilon_\theta\overline{\int_Q\Phi^\varepsilon_\theta}=\bigl\langle H^\varepsilon_\theta,\Phi^\varepsilon_\theta\bigr\rangle\equiv
\biggl\langle H^\varepsilon_\theta,\Phi^\varepsilon_\theta-\int_Q\Phi^\varepsilon_\theta\biggr\rangle,
\]
so that 
\begin{align}
\int_QA\nabla(e_{\varepsilon\theta}\Phi^\varepsilon_\theta)&\cdot{\nabla\biggl\{e_{\varepsilon\theta}\biggl(R^\varepsilon_\theta-\int_QR^\varepsilon_\theta\biggr)\biggr\}}=\overline{\biggl\langle H^\varepsilon_\theta,\Phi^\varepsilon_\theta-\int_Q\Phi^\varepsilon_\theta\biggr\rangle}\nonumber\\[0.5em]
&-\overline{\int_QR^\varepsilon_\theta}\biggl(\int_QA\nabla(e_{\varepsilon\theta}\Phi^\varepsilon_\theta)\cdot{\nabla e_{\varepsilon\theta}}+\varepsilon^2\int_Q\Phi^\varepsilon_\theta\biggr)=\overline{\biggl\langle H^\varepsilon_\theta,\Phi^\varepsilon_\theta-\int_Q\Phi^\varepsilon_\theta\biggr\rangle},
\label{Phi_eq}
\end{align}
where the values of the functional $H^\varepsilon_\theta$ are chosen accordingly. In the last equality in (\ref{Phi_eq}) we use the fact that 
\[
\int_QA\nabla(e_{\varepsilon\theta}\Phi^\varepsilon_\theta)\cdot{\nabla e_{\varepsilon\theta}}+\varepsilon^2\int_Q\Phi^\varepsilon_\theta=0,
\]
by setting the unity as a test function in the integral formulation of (\ref{Phi_eq0}) and recalling that ({\it cf.} (\ref{N_equation}))
\[
\int_Q AN_j=0,\qquad j=1,2,\dots, d.
\]

Combining (\ref{penultimate}), (\ref{interm}) and (\ref{Phi_eq}) yields 
\begin{align}
&\int_QA\nabla(e_{\varepsilon\theta}R^\varepsilon_\theta)\cdot{\nabla(e_{\varepsilon\theta}R^\varepsilon_\theta)}+\varepsilon^2\biggl\vert\int_QR^\varepsilon_\theta\biggr\vert^2=
\int_Q\biggl[F-c_\theta\Bigl(\theta\cdot A\nabla(N_j\theta_j)+({\rm i}\varepsilon N_j\theta_j+1)\theta\cdot A\theta
\nonumber\\[0.5em]
%-\theta\cdot A\theta
&+e_{\varepsilon\theta}N_j\theta_j\theta\cdot{A\nabla e_{\varepsilon\theta}}+1\Bigr)\biggr]\overline{\biggl(R^\varepsilon_\theta-\int_QR^\varepsilon_\theta\biggr)}
%\nonumber\\[0.5em]
%&
+\overline{\biggl\langle H^\varepsilon_\theta,\Phi^\varepsilon_\theta-\int_Q\Phi^\varepsilon_\theta\biggr\rangle}.
%-c_\theta\int_Qe_{\varepsilon\theta}N_j\theta_jA\theta\cdot{\biggl(R^\varepsilon_\theta-\int_QR^\varepsilon_\theta\biggr)\nabla e_{\varepsilon\theta}}
%\nonumber\\[0.5em]
%&
%%-c_\theta\int_QR^\varepsilon_\theta\int_Qe_{\varepsilon\theta}N_j\theta_j\theta\cdot{A\nabla e_{\varepsilon\theta}}
%-c_\theta)\overline{R^\varepsilon_\theta}
\label{ultimate}
\end{align}
%%%the bound (\ref{R_average_est}), 

\begin{lemma} 
The last term on the right-hand side of (\ref{ultimate}) is bounded, uniformly in $\varepsilon,$ $\theta:$
\[
\Biggl\vert\biggl\langle H^\varepsilon_\theta,\Phi^\varepsilon_\theta-\int_Q\Phi^\varepsilon_\theta\biggr\rangle\Biggr\vert\le C\Vert F\Vert_{L^2(Q)},\qquad C>0.
\]
\end{lemma}
\begin{proof}
%%Consider the pair $\bigl(e_{\varepsilon\theta}\Phi^\varepsilon_\theta, \nabla(e_{\varepsilon\theta}\Phi^\varepsilon_\theta)\bigr)$ that solves (\ref{Phi_eq0}). 
%%%%Notice first that for any gradient $\nabla\Phi^\varepsilon_\theta$ one has
%%\[
%%\nabla\Phi^\varepsilon_\theta=\overline{e_{\varepsilon\theta}}\bigl(\nabla(e_{\varepsilon\theta}\Phi^\varepsilon_\theta)-{\rm i}\varepsilon\Phi^\varepsilon_\theta\theta\bigr),
%%\]
%%see the discussion following Definition \ref{Sobolev_definition}.
Notice that 
\begin{align}
\biggl\langle H^\varepsilon_\theta, \Phi^\varepsilon_\theta-\int_Q\Phi^\varepsilon_\theta\biggr\rangle&=
\int_Q\Bigl(F+{\rm i}\nabla\cdot \bigl({\rm i} N_j\theta_jA\theta\bigr)c_\theta+{\rm i}\theta\cdot A\nabla({\rm i} N_j\theta_j)c_\theta
\nonumber\\[0.5em]
&-{\rm i}\varepsilon N_j\theta_j\theta\cdot A\theta c_\theta-\theta\cdot A\theta c_\theta-c_\theta\Bigr)\biggl(\Phi^\varepsilon_\theta-\int_Q\Phi^\varepsilon_\theta\biggr)+c_\theta\int_QN_j\theta_jA\theta\cdot\nabla\Phi^\varepsilon_\theta,
\label{H_test}
\end{align}
where the second term is re-written using (\ref{grad_Phi}):
\[
\int_QN_j\theta_jA\theta\cdot\nabla\Phi^\varepsilon_\theta=\int_Q\overline{e_{\varepsilon\theta}}N_j\theta_jA\theta\cdot\nabla(e_{\varepsilon\theta}\Phi^\varepsilon_\theta)-{\rm i}\varepsilon\int_Q\overline{e_{\varepsilon\theta}}N_j\theta_jA\theta\cdot\theta\biggl(\Phi^\varepsilon_\theta-\int_Q\Phi^\varepsilon_\theta\biggr).
\]
Applying the H\"{o}lder inequality to both terms on the right-hand side of (\ref{H_test}), using the Poincar\'{e} inequality (\ref{Poincare}) for $\Phi^\varepsilon_\theta,$ and taking into the account the bound (\ref{gradPhi_est}) yields the required estimate.
\end{proof}

Combining the above lemma, the Poincar\'{e} inequality (\ref{Poincare}) for $R^\varepsilon_\theta$
%(Proposition \ref{Poincare_estimate}) 
and H\"{o}lder inequality for the first term on the right-hand side of (\ref{ultimate}), we obtain the uniform bound
\begin{equation}
\bigl\Vert\sqrt{A}\nabla(e_{\varepsilon\theta}R^\varepsilon_\theta)\bigr\Vert_{[L^2(Q)]^d}\le C\Vert F\Vert_{L^2(Q)}.
\label{grad_est}
\end{equation}
%%%and the definition (\ref{c_condition}) of the constants $c_\theta$ yields a uniform estimate for $\bigl\Vert\sqrt{A}\nabla(e_{\varepsilon\theta}R^\varepsilon_\theta)\bigr\Vert_{L^2(Q)}$ in terms of $\Vert F\Vert_{L^2(Q)}.$ 
 
 Finally, the bound (\ref{grad_est}) combined with (\ref{Poincare})
 %Proposition \ref{Poincare_estimate} 
 implies the first estimate in (\ref{target_est}), whereas the same bound and equation (\ref{ultimate}) implies the second estimate in (\ref{target_est}).
 This completes the proof of the theorem.
%%%Once again invoking Proposition \ref{Poincare_estimate} 
%we obtain the first inequality in (\ref{target_est}) and invoking 
%%%and (\ref{R_average_est}) implies the first and second inequality in (\ref{target_est}), respectively.
\end{proof}

Note that in the case $\theta=0$ the equality (\ref{penultimate}) takes the form
\begin{equation}
\int_QA\nabla R^\varepsilon_0\cdot{\nabla R^\varepsilon_0}
+\varepsilon^2\biggl\vert\int_QR^\varepsilon_0\biggr\vert^2=\int_QF\overline{R^\varepsilon_0},
\label{kappa0}
\end{equation}
and taking into account (\ref{Poincare}) with $\varkappa=0,$ we obtain 
\begin{equation*}
\bigl\Vert\nabla R^\varepsilon_0\bigr\Vert_{[L^2(Q)]^d}\le C\Vert F\Vert_{L^2(Q)}.
%\label{Rest_zero}
\end{equation*}
The last estimate implies the first bound in (\ref{target_est}) by (\ref{Poincare}) with $\varkappa=0$ and the second bound in (\ref{target_est}) by using (\ref{kappa0}) once again.

\begin{corollary}
\label{cor}
The following estimate holds uniformly in $\varepsilon>0,$ $\theta\in\varepsilon^{-1}Q',$ $F\in L^2(Q):$
%[-\pi/\varepsilon,\pi/\varepsilon):$
\[
\Vert U_\theta^\varepsilon-c_\theta\Vert_{L^2(Q)}\le C\varepsilon\Vert F\Vert_{L^2(Q)}.
\]
\end{corollary}

\section{Conclusion of the convergence estimate (\ref{main_est})}
\label{conclusion}

Here we estimate the error incurred by using the approximation $U_\theta^\varepsilon$ in (\ref{strong_form}).
\begin{proposition}
\label{z_est_prop}
The difference $z^\varepsilon_\theta:=u_\theta^\varepsilon-U_\theta^\varepsilon$ satisfies the estimate
%the following estimate with $C>0:$
\[
\Vert z^\varepsilon_\theta\Vert_{L^2(Q)}\le C\varepsilon\Vert F\Vert_{L^2(Q)},\quad C>0,\qquad\forall\ \varepsilon>0,\ \theta\in\varepsilon^{-1}Q',\ F\in L^2(Q).
\]
\end{proposition}
\begin{proof}
It follows from (\ref{strong_form}), (\ref{U_def}), (\ref{c_condition}), (\ref{N_equation}), (\ref{R_equation}), by a direct calculation, that 
\begin{equation}
-\varepsilon^{-2}\overline{e_{\varepsilon\theta}}\nabla\cdot A\nabla(e_{\varepsilon\theta}z^\varepsilon_\theta)+z^\varepsilon_\theta=-{\rm i}\varepsilon N_j\theta_jc_\theta-\varepsilon^2\biggl(R^\varepsilon_\theta-\int_QR^\varepsilon_\theta\biggr).
%\label{z_est}
\label{z_eq}
\end{equation}
%the subspace consisting of gradients of zero.
%\end{remark}
%It follows from (\ref{z_est}) that 
In particular, using $z^\varepsilon_\theta$ as a test function in (\ref{z_eq}), we obtain
\[
\varepsilon^{-2}\int A\nabla(e_{\varepsilon\theta}z^\varepsilon_\theta)\cdot{\nabla e_{\varepsilon\theta}z^\varepsilon_\theta}+\int_Q\vert z^\varepsilon_\theta\vert^2=-{\rm i}\varepsilon c_\theta\theta_j\int_QN_j\overline{z^\varepsilon_\theta}-\varepsilon^2\int_Q\biggl(R^\varepsilon_\theta-\int_QR^\varepsilon_\theta\biggr)\overline{z^\varepsilon_\theta},
\]
and hence
\begin{align*}
\Vert z^\varepsilon_\theta\Vert_{L^2(Q)}^2&\le\varepsilon\vert c_\theta\vert\vert\theta\vert\Vert N\Vert_{[L^2(Q)]^d}\Vert z^\varepsilon_\theta\Vert_{L^2(Q)}+\varepsilon^2C_{\rm P}\bigl\Vert\nabla(e_{\varepsilon\theta}R^\varepsilon_\theta)\bigr\Vert_{[L^2(Q)]^d}\Vert z^\varepsilon_\theta\Vert_{L^2(Q)}\nonumber\\[0.5em]
&\le \varepsilon\Bigl(\vert c_\theta\vert\vert\theta\vert\Vert N\Vert_{[L^2(Q)]^d}+\varepsilon C\bigl\Vert\sqrt{A}\nabla(e_{\varepsilon\theta}R^\varepsilon_\theta)\bigr\Vert_{[L^2(Q)]^d}\Bigr)\Vert z^\varepsilon_\theta\Vert_{L^2(Q)},
\end{align*}
where we use Proposition \ref{Poincare} once again and the fact that $A$ is uniformly positive. The claim follows, by virtue of the formula (\ref{c_condition}) and the estimate (\ref{grad_est}).
\end{proof}

Combining Corollary \ref{cor}  and Proposition \ref{z_est_prop} 
%and the bound (\ref{Rest_zero}) 
concludes the proof of Theorem \ref{main_theorem}.

%%Imposing the condition that the right-hand side of (\ref{R_equation}) is orthogonal to $\psi_\varkappa,$ we obtain the uniform estimate 
%%\begin{align}
%%\Vert R^\varepsilon_\theta\Vert_{L^2(Q)}=\Vert e_{\varepsilon\theta}R^\varepsilon_\theta\Vert_{L^2(Q)}&\le
 %\frac{1}{\lambda_1^{\rm min}} 
%%c\bigl\Vert\nabla\bigl(e_{\varepsilon\theta}R^\varepsilon_\theta\bigr)\bigr\Vert_{L^2(Q)}\nonumber\\[0.3em]
%%&\le\bigl\Vert F-c_\theta-
%%A\bigl(\nabla(e_{\varepsilon\theta}\widetilde{N}^{\varepsilon\theta}\cdot\theta)+e_{\varepsilon\theta}\theta\bigr)\cdot{e_{\varepsilon\theta}}\theta c_\theta\bigr\Vert_{L^2(Q)},
%%\label{R_est}
%%\end{align}
%%where $c^{-1}$ is the minimum of the second eigenvalues of the operators $A_\varkappa$ over all $\varkappa\in[-\pi,\pi).$

%%The above orthogonality condition reads
%%\[
%%\Biggl(\int_Q\Bigl\{A\bigl(\nabla(e_{\varepsilon\theta} \widetilde{N}^{\varepsilon\theta}\cdot\theta)+e_{\varepsilon\theta}\theta\bigr)\overline{e_{\varepsilon\theta}}\theta +1\Bigr\}\overline{\psi_{\varepsilon\theta}}\Biggr)c_\theta=\int_Q F\psi_{\varepsilon\theta}
%%\]
%%from which we find the constant $c_\theta.$
 
%$\lambda_2(\varkappa)\}$
%\ 
 
%\leftline{\bf Estimates}

%\

%We have thus reduced the problems to the analysis of the behaviour of $\lambda_\varkappa,$ $\psi_\varkappa$ with respect to $\varkappa.$

%\begin{remark}
%We may get a more compact and convenient approximation if we replace the leading term $c_\theta$ in (\ref{U_def}) by $c_\theta\psi_{\varepsilon\theta}.$

%\end{remark}

%\

%\subsection{Outer region: $\vert\theta\vert\ge1/\sqrt{\varepsilon}$}

%%%%%\section{Generalisations}

\section*{Acknowledgments}
KC is grateful for the support of
%has been supported by
%the Leverhulme Trust (Grant RPG--167 ``Dissipative and non-self-adjoint problems'') and
the Engineering and Physical Sciences Research Council: Grant EP/L018802/2 ``Mathematical foundations of metamaterials: homogenisation, dissipation and operator theory''. We are also grateful to Alexander Kiselev and Igor Vel\v{c}i\'{c} for helpful discussions.
%}


\begin{thebibliography}{9}

%\bibitem{BLP}
%Bensoussan, A., Lions, J.-L., and Papanicolaou, G. C., 1978. {\it Asymptotic Analysis for Periodic Structures,} North-Holland

 \bibitem{Birman_Solomjak}
Birman, M. S., and Solomjak, M. Z., 1986. {\it Spectral Theory of Self-Adjoint Operators in Hilbert Space}, D. Reidel Publishing Company.

\bibitem{BS}
Birman, M. Sh., and  Suslina, T. A.,  2004. Second order periodic differential operators. Threshold properties and homogenisation. {\it St. Petersburg. Math. J.} {\bf 15} (5), 639--714.

\bibitem{ChCoARMA}
Cherednichenko, K. D., Cooper, S., 2016. Resolvent estimates for high-contrast elliptic problems with periodic coefficients. {\it Archive for Rational Mechanics and Analysis} {\bf 219} (3), 1061--1086.

\bibitem{ChDO_20} Cherednichenko, K., D'Onofrio, S., 2020. Operator-norm homogenisation estimates for the system of Maxwell equations on periodic singular structures, {\it arXiv: 1811.08980}.

\bibitem{Gelfand}
Gel'fand, I. M., 1950. Expansion in characteristic functions of an equation with periodic coefficients. (Russian) {\it Doklady Akad. Nauk SSSR (N.S.)} {\bf 73}, 1117--1120.

%\bibitem{JKO} 
%Jikov, V. V., Kozlov, S. M., and Oleinik, O. A. 1994. {\it Homogenization of differential operators and 
%integral functionals}, Springer.

\bibitem{Sevost'yanova}
Sevost'yanova, E. V., 1981. Asymptotic expansion of the solution of a second-order elliptic equation with periodic rapidly oscillating coefficients. (Russian) Mat. Sb. (N.S.) {\bf 115}(2) 204--222.

\bibitem{Zhikov1989}
Zhikov, V. V., 1989. Spectral approach to asymptotic problems in diffusion. {\it Diff. Equations} {\bf 25}, 33--39.

\bibitem{Zhikov2000}
Zhikov, V. V., 2000. On an extension of the method of two-scale convergence and its applications, {\it Sb. Math.} 
{\bf 191}(7), 973--1014.

\bibitem{Zhikov_Note}
Zhikov, V. V., 2005. A note on Sobolev spaces, {\it Journal of Mathematical Sciences} {\bf 129}(1), 3593--3595.

\bibitem{Zhikov_Pastukhova_Bloch}
Zhikov, V. V., Pastukhova S. E., 2016. Bloch principle for elliptic differential operators with periodic coefficients, {\it Russian Journal of Mathematical Physics} {\bf 23}(2), 257--277.



\end{thebibliography}
\end{document}